\documentclass[11pt]{article}
\usepackage[T1]{fontenc}
\usepackage[utf8]{inputenc}
\usepackage[english]{babel}
\usepackage{lmodern}
\usepackage{microtype}
\usepackage{amsmath,amssymb,amsfonts,amsthm,mathtools}
\usepackage[shortlabels]{enumitem}
\usepackage{geometry}
\usepackage[colorlinks=true,linkcolor=blue,citecolor=blue,urlcolor=blue]{hyperref}
\newtheorem{theorem}{Theorem}[section]
\newtheorem{lemma}[theorem]{Lemma}

\newtheorem{proposition}[theorem]{Proposition}
\newtheorem{corollary}[theorem]{Corollary}

\newtheorem{remark}[theorem]{Remark}

\newtheorem{conjecture}[theorem]{Conjecture}
\newtheorem{question}[theorem]{Question}
\newcommand{\eps}{\varepsilon}
\newcommand{\cA}{\mathcal{A}}
\newcommand{\cC}{\mathcal{C}}
\newcommand{\cH}{\mathcal{H}}
\newcommand{\cM}{\mathcal{M}}
\newcommand{\cP}{\mathcal{P}}
\newcommand{\cS}{\mathcal{S}}
\newcommand{\cU}{\mathcal{U}}
\newcommand{\cW}{\mathcal{W}}
\newcommand{\Emb}{\operatorname{Emb}}
\newcommand{\Hom}{\operatorname{Hom}}
\newcommand{\dist}{\operatorname{dist}}
\newcommand{\rank}{\operatorname{rank}}
\newcommand{\im}{\operatorname{im}}
\newcommand{\codim}{\operatorname{codim}}
\title{Compactness of abundance in asymmetric hypergraph removal lemmas}
\author{Shuang Sun\thanks{Email: \texttt{chocolatesun@sjtu.edu.cn}},\quad Yan Wang\thanks{Email: \texttt{yan.w@sjtu.edu.cn}},\quad Yuyao Yang\thanks{Email: \texttt{alaia\_y@sjtu.edu.cn}},\quad and Jiasheng Zeng\thanks{Email: \texttt{jasonzeng@mail.ustc.edu.cn}}}
\date{\today}

\begin{document}
\maketitle

\begin{abstract}
Fix an integer $r\ge 2$ and a finite simple $r$-uniform hypergraph $F$ with at least one edge and no isolated vertices. An $n$-vertex $r$-graph is $\eps$-far from being $F$-free if at least $\eps n^r$ edges must be deleted to destroy every copy of $F$. A finite $r$-graph $H$ is $F$-abundant if there are constants $c,C>0$ such that every sufficiently large $\eps$-far host contains at least $c\eps^C n^{v(H)}$ labelled copies of $H$. A family is $F$-abundant when one member has this lower bound in each host, although the member may depend on the host and on $\eps$, while $c$ and $C$ are common to the family. We prove that every abundant family contains an abundant member. We prove the analogous coloured theorem for $F$-partite hosts containing edge-disjoint part-respecting copies of $F$ such that every vertex lies in at least $\eps n^{r-1}$ of them. The case $r=2$ yields the coloured and uncoloured graph compactness theorems, answers Question 5.2 of Gir\~ao, Hurley, Illingworth and Michel, and proves their Conjecture 5.1 [J. Lond. Math. Soc., 2024]. We also obtain an explicit bound $\lfloor 2r(C+1)\rfloor$ for the order of the non-isolated core of a selected witness. Moreover, we give several applications. For example, we construct translation-invariant linear systems from abundant coloured hypergraphs, obtain a square-root bound for an equation associated with a cycle of bounded length, give a one-sided tester based on one fixed graph when distance from the property gives a polynomial lower bound on distance from being $F$-free, and prove that no algorithm decides whether a family of finite simple graphs enumerated by a Turing machine is $K_3$-abundant.
\end{abstract}

\section{Introduction}
\label{sec:intro}

Let $F$ be a fixed graph. A graph is $F$-free if it contains no copy of $F$, and an $n$-vertex graph $G$ is $\eps$-far from being $F$-free if at least $\eps n^2$ edges must be deleted from $G$ to make it $F$-free. The triangle removal lemma of Ruzsa and Szemer\'edi \cite{RS78} states that for every $\eps>0$ there is $\delta=\delta(\eps)>0$ such that every sufficiently large $n$-vertex graph that is $\eps$-far from being triangle-free contains at least $\delta n^3$ triangles. The lemma arose from the study of the $(6,3)$-problem and gives a graph-theoretic proof of Roth's theorem on three-term arithmetic progressions \cite{Roth53,RS78}. Erd\H{o}s, Frankl and R\"odl extended it to every fixed graph \cite{EFR86}. More precisely, for every graph $F$ and every $\eps>0$ there is $\delta=\delta(F,\eps)>0$ such that every sufficiently large $n$-vertex graph that is $\eps$-far from being $F$-free contains at least $\delta n^{v(F)}$ labelled copies of $F$. Regularity proofs imply that $\delta^{-1}$ is a tower of twos whose height is polynomial in $1/\eps$, and Fox reduced the height to $O_F(\log(1/\eps))$ \cite{Fox11}. Behrend's construction and the construction of Ruzsa and Szemer\'edi show that $\delta$ cannot admit a polynomial lower bound when $F=K_3$ \cite{Beh46,RS78}. Alon proved more generally that there are constants $c_F,C_F>0$ for which one may take $\delta\ge c_F\eps^{C_F}$ if and only if $F$ is bipartite \cite{Alon02}. The graph removal lemma has many applications in extremal graph theory and additive combinatorics \cite{ConlonFox13,Zhao23}. It is also a basic tool in property testing, which asks whether a large graph has a given property or is far from having it by examining only a small random sample of its vertices \cite{GGR98,Alon02,AlonShapira08}.

The graph removal lemma extends to uniform hypergraphs. For $r\ge2$, an $r$-graph is a finite simple $r$-uniform hypergraph, and an $n$-vertex $r$-graph $G$ is $\eps$-far from being $F$-free if at least $\eps n^r$ edges must be deleted from $G$ to destroy every copy of $F$. The hypergraph removal lemma states that for every fixed $r$-graph $F$ and every $\eps>0$ there is $\delta=\delta(F,\eps)>0$ such that every sufficiently large $n$-vertex $r$-graph that is $\eps$-far from being $F$-free contains at least $\delta n^{v(F)}$ labelled copies of $F$. R\"odl, Nagle, Skokan, Schacht and Kohayakawa developed one version of the hypergraph regularity and counting method used to prove the hypergraph removal lemma \cite{RodlEtAl05}. Gowers independently developed another approach \cite{Gowers07}. Tao later gave a self-contained proof of a slightly stronger variant of the hypergraph removal lemma \cite{Tao06}. The general proofs of hypergraph removal give extremely weak dependence of $\delta$ on $\eps$. Gishboliner and Shapira determined exactly when polynomial dependence is possible, proving that there are constants $c_F,C_F>0$ such that one may take $\delta\ge c_F\eps^{C_F}$ if and only if $F$ is $r$-partite \cite{GishbolinerShapira25}. By encoding arithmetic configurations as copies of fixed uniform hypergraphs, the lemma gives combinatorial proofs of Szemer\'edi's theorem and its multidimensional extension \cite{Szemeredi75,RodlEtAl05,Gowers07,Tao06}. Shapira used a coloured form of the hypergraph removal lemma to prove a removal lemma for systems of linear equations \cite{Shapira10}.

Polynomial removal bounds have also been studied for arithmetic configurations over finite fields \cite{FL17,FLS18}, ordered graphs \cite{GishbolinerSimic24,GT22}, matrices \cite{ABE20}, posets \cite{FeketeKun23}, and graph properties \cite{GS21}. In this paper, we study asymmetric removal, where the hypergraph $F$ that cannot be removed and the hypergraph $H$ that is counted may be different.

We first give the definitions needed throughout the paper. For $r$-graphs $H$ and $G$, a homomorphism from $H$ to $G$ is a map $\phi\colon V(H)\to V(G)$ such that $\phi(e)\in E(G)$ for every $e\in E(H)$. An embedding is an injective homomorphism. We write $\Hom(H,G)$ and $\Emb(H,G)$ for the numbers of homomorphisms and embeddings from $H$ to $G$. Copies are not required to be induced.

Fix an integer $r\ge2$ and an $r$-graph $F$ with $e(F)>0$. An $n$-vertex $r$-graph $G$ is $\eps$-far from being $F$-free if at least $\eps n^r$ edges must be deleted from $G$ to make it $F$-free, as in the uniform-hypergraph definition of Gishboliner and \v{S}imi\'c \cite[p.~2]{GishbolinerSimic24}. An $r$-graph $H$ is $F$-abundant if there are constants $c,C>0$ such that, for every $0<\eps\le1$, there is $n_0(\eps)$ for which every $\eps$-far $n$-vertex $r$-graph $G$ with $n\ge n_0(\eps)$ satisfies $\Emb(H,G)\ge c\eps^C n^{v(H)}$. For graphs, this is the definition of $F$-abundance in \cite[Definition~2.1]{GHIM}, with copies counted by embeddings.

An $F$-colouring of an $r$-graph $G$ is a homomorphism $\kappa_G\colon G\to F$. The sets $V_a(G)=\kappa_G^{-1}(a)$ for $a\in V(F)$ form an $F$-partition of $G$, and an $r$-graph with such a partition is $F$-partite. An $F$-coloured $r$-graph is a pair $(H,\sigma)$ in which $\sigma\colon H\to F$ is an $F$-colouring. These definitions extend \cite[Definition~2.2]{GHIM} from graphs to $r$-graphs. If $G$ has a fixed $F$-partition, we write $\Hom_\sigma(H,G)$ and $\Emb_\sigma(H,G)$ for the numbers of homomorphisms and embeddings $\phi\colon H\to G$ satisfying $\kappa_G\circ\phi=\sigma$.

An $n$-vertex $F$-partite $r$-graph $G$ is uniformly $\eps$-far from being $F$-free if there is a collection $\cC$ of embeddings $\psi\colon F\to G$ such that the sets $\psi(E(F))$, for $\psi\in\cC$, are pairwise disjoint, $\kappa_G\circ\psi=\operatorname{id}_{V(F)}$ for every $\psi\in\cC$, and every vertex of $G$ belongs to $\psi(V(F))$ for at least $\eps n^{r-1}$ members $\psi\in\cC$. For graphs, this is the definition preceding \cite[Lemma~2.4]{GHIM}. An $F$-coloured $r$-graph $(H,\sigma)$ is $F$-abundant if there are constants $c,C>0$ such that, for every $0<\eps\le1$, there is $n_0(\eps)$ for which every uniformly $\eps$-far $n$-vertex $F$-partite $r$-graph $G$ with $n\ge n_0(\eps)$ satisfies $\Emb_\sigma(H,G)\ge c\eps^C n^{v(H)}$. For graphs, this is \cite[Definition~2.6]{GHIM}.

We next state the definition for families used by Gir\~ao, Hurley, Illingworth, and Michel. For a graph $H$ and a graph $G$, let $N(H,G)=\Emb(H,G)/|\operatorname{Aut}(H)|$. A family $\cU$ of graphs is triangle-abundant if there is $C>0$ such that, for every $0<\eps\le1$, there is $n_0(\eps)$ for which every $n$-vertex graph $G$ with $n\ge n_0(\eps)$ that has at least $\eps n^2$ pairwise edge-disjoint triangles contains some $H\in\cU$ satisfying $N(H,G)\ge\eps^C n^{v(H)}$. If $(H,\sigma)$ is $K_3$-coloured, let $\operatorname{Aut}(H,\sigma)$ be the set of automorphisms $\alpha$ of $H$ satisfying $\sigma\circ\alpha=\sigma$, and let $N_\sigma(H,G)=\Emb_\sigma(H,G)/|\operatorname{Aut}(H,\sigma)|$. A family $\cH$ of $K_3$-coloured graphs is triangle-abundant if there is $C>0$ such that, for every $0<\eps\le1$, there is $n_0(\eps)$ for which every $n$-vertex graph $G$ with $n\ge n_0(\eps)$ that has a $K_3$-partition and a collection of at least $\eps n^2$ pairwise edge-disjoint triangles whose three vertices lie in the three prescribed parts contains some $(H,\sigma)\in\cH$ satisfying $N_\sigma(H,G)\ge\eps^C n^{v(H)}$. This is the family definition preceding \cite[Question~5.2]{GHIM}, with the sufficiently large order convention used in \cite[Definitions~2.1 and~2.6]{GHIM}.

Gir\~ao, Hurley, Illingworth, and Michel observed that all known proofs that a graph $H$ is $F$-abundant proceed by finding an $F$-colouring $\sigma$ for which $(H,\sigma)$ is $F$-abundant. They asked whether this is necessary and whether the corresponding statement holds for families.

\begin{conjecture}[Gir\~ao, Hurley, Illingworth, and Michel \cite{GHIM}]\label{conjGHIM}
Let $F$ and $H$ be finite simple graphs with $e(F)>0$. The graph $H$ is $F$-abundant if and only if there is an $F$-colouring $\sigma$ of $H$ such that $(H,\sigma)$ is $F$-abundant.
\end{conjecture}

\begin{question}[Gir\~ao, Hurley, Illingworth, and Michel \cite{GHIM}]\label{quesGHIM}
Is a family of graphs triangle-abundant if and only if it contains a triangle-abundant graph? Is a family of $K_3$-coloured graphs triangle-abundant if and only if it contains a triangle-abundant $K_3$-coloured graph?
\end{question}

We extend the above definitions to families of $r$-graphs. A possibly infinite family $\cU$ of $r$-graphs is $F$-abundant if there are constants $c,C>0$ such that, for every $0<\eps\le1$, there is $n_0(\eps)$ for which every $\eps$-far $n$-vertex $r$-graph $G$ with $n\ge n_0(\eps)$ contains some $H\in\cU$ satisfying $\Emb(H,G)\ge c\eps^C n^{v(H)}$. A possibly infinite family $\cH$ of $F$-coloured $r$-graphs is $F$-abundant if there are constants $c,C>0$ such that, for every $0<\eps\le1$, there is $n_0(\eps)$ for which every uniformly $\eps$-far $n$-vertex $F$-partite $r$-graph $G$ with $n\ge n_0(\eps)$ contains some $(H,\sigma)\in\cH$ satisfying $\Emb_\sigma(H,G)\ge c\eps^C n^{v(H)}$. In each definition, the member of the family may depend on $G$ and $\eps$. The constants $c$ and $C$ are fixed for the whole family, while $n_0(\eps)$ may depend on $\eps$ but not on $G$ or on the selected member.

Our main theorem gives an affirmative answer to the $r$-uniform analogue of Question~\ref{quesGHIM} for every $r\ge2$.

\begin{theorem}\label{thmcompact}
Let $r\ge2$, and let $F$ be a finite $r$-graph with $e(F)>0$ and no isolated vertices. A possibly infinite family of $F$-coloured finite $r$-graphs is $F$-abundant if and only if it contains an $F$-abundant member. A possibly infinite family of finite $r$-graphs is $F$-abundant if and only if it contains an $F$-abundant member.
\end{theorem}

Taking $r=2$ in Theorem~\ref{thmcompact} gives the following graph statement.
\begin{corollary}\label{corgraphcompact}
Let $F$ be a finite simple graph with $e(F)>0$ and no isolated vertices. A possibly infinite family of $F$-coloured finite graphs is $F$-abundant if and only if it contains an $F$-abundant member. A possibly infinite family of finite graphs is $F$-abundant if and only if it contains an $F$-abundant member.
\end{corollary}

For an $r$-graph $H$, let $H^\circ$ be obtained from $H$ by deleting all isolated vertices. If $(H,\sigma)$ is $F$-coloured, let $\sigma^\circ$ be the restriction of $\sigma$ to $H^\circ$. The proof of Theorem~\ref{thmcompact} also gives the following bound.
We call $H^\circ$ the non-isolated core of $H$.

\begin{theorem}\label{thmboundedcore}
Let $r\ge2$, and let $F$ be a finite $r$-graph with $e(F)>0$ and no isolated vertices. Suppose that a family is $F$-abundant with constants $c,C>0$. If the family consists of $F$-coloured $r$-graphs, it contains an $F$-abundant member $(H,\sigma)$ satisfying $v(H^\circ)\le\lfloor2r(C+1)\rfloor$. If the family consists of $r$-graphs, it contains an $F$-abundant member $H$ satisfying the same bound.
\end{theorem}

The number $v(H)$ cannot be bounded using only $r$ and $C$. Let $E_r$ consist of one edge, let $I_s$ consist of $s$ isolated vertices, and let $H_s=E_r\mathbin{\dot\cup}I_s$. Each $H_s$ is $E_r$-abundant with a value of $C$ independent of $s$, while $v(H_s)$ tends to infinity. Thus the conclusion of Theorem~\ref{thmboundedcore} cannot in general be strengthened by replacing $v(H^\circ)$ with $v(H)$.

Theorem~\ref{thmcompact}, together with a reduction to the non-isolated core of $F$, also shows that if $H$ is $F$-abundant, then one fixed $F$-colouring of $H$ is $F$-abundant, even when $F$ has isolated vertices.

\begin{corollary}\label{corfixedcolour}
Let $r\ge2$, and let $F$ be a finite $r$-graph with $e(F)>0$. An $r$-graph $H$ is $F$-abundant if and only if there is an $F$-colouring $\sigma$ of $H$ such that $(H,\sigma)$ is $F$-abundant.
\end{corollary}

\begin{corollary}\label{corGHIM}
Question~\ref{quesGHIM} has an affirmative answer, and Conjecture~\ref{conjGHIM} holds.
\end{corollary}

We explain how Corollary~\ref{corGHIM} follows from the preceding statements. If an $n$-vertex graph is $\eps$-far from being triangle-free, a maximal collection of pairwise edge-disjoint triangles has at least $\eps n^2/3$ members, since deleting the edges in the collection destroys every triangle. Conversely, a collection of $\eps n^2$ pairwise edge-disjoint triangles requires at least $\eps n^2$ edge deletions to destroy every triangle. For a graph with a $K_3$-partition, uniform $\eps$-farness gives at least $\eps n^2/3$ edge-disjoint triangles whose vertices lie in the prescribed parts. Conversely, starting with $\eps n^2$ such triangles and repeatedly deleting all remaining triangles through a vertex that belongs to fewer than $\eps n/2$ of them leaves at least $\eps n^2/2$ triangles. The vertices in the remaining triangles span a graph of order $n'\ge(\eps/2)^{1/2}n$ that is uniformly $\eps/2$-far from being triangle-free. Finally, for each fixed graph selected by Theorem~\ref{thmcompact}, the two counts differ by $|\operatorname{Aut}(H)|$, or by $|\operatorname{Aut}(H,\sigma)|$ when a colouring is fixed. These constant factors can be absorbed by changing the exponent of $\eps$.

Theorem~\ref{thmcompact} is not a consequence of applying a finiteness argument to the family. In the definition for a family, the graph that satisfies the required inequality may vary with $G$ and $\eps$, and neither its order nor its isomorphism type is bounded. Generalized Tur\'an problems show that the number of copies of one graph can depend sensitively on the exclusion of another graph \cite{AlonShikhelman16,GishbolinerShapira20}. The theorem shows that the polynomial lower bound nevertheless forces one fixed member of the family to satisfy the required inequality for every sufficiently large graph in the definition.

There have been many related works on the study of asymmetric removal. Csaba proved a single-exponential conditional triangle removal lemma for tripartite graphs \cite[Theorem~4.2]{Csaba24}. It follows that every sufficiently large $n$-vertex graph that is $\eps$-far from being triangle-free contains at least $2^{-\operatorname{poly}(1/\eps)}n^5$ embeddings of $C_5$. Gishboliner, Shapira, and Wigderson proved that, whenever $1\le k<\ell$, every sufficiently large $n$-vertex graph that is $\eps$-far from being $C_{2k+1}$-free contains at least $c_\ell\eps^{4\ell+2}n^{2\ell+1}$ embeddings of $C_{2\ell+1}$ \cite[Theorem~1.3]{GSW}. Gir\~ao, Hurley, Illingworth, and Michel constructed $K_t$-abundant graphs of chromatic number $t$ for every $t\ge4$ \cite[Theorem~1.1]{GHIM} and related the abundance of $F$-coloured graphs to systems of linear equations \cite[Section~4]{GHIM}.

The bound on $v(H^\circ)$ is also useful in other problems. Varnavides-type estimates and arithmetic removal results provide examples for passing from dense sets to many arithmetic configurations \cite{Varnavides59,Green05}. Section~\ref{secapplications} associates a translation-invariant linear system with an $F$-coloured abundant $r$-graph by choosing distinct numbers for the vertices of $F$ and using the corresponding Vandermonde matrices on the edges. It proves that abundance gives a polynomial lower bound for the number of solutions over finite fields and over the integers. If a family is $F$-abundant with exponent $C$, Theorem~\ref{thmboundedcore} gives a system with at most $\binom{\lfloor2r(C+1)\rfloor}{r}$ variables, and an explicit rank condition guarantees that this system contains an equation. For an $F$-abundant family of cycles, it gives a cycle of length at most $\lfloor4(C+1)\rfloor$ whose associated equation has genus at least two and therefore satisfies the square-root bound of Gir\~ao, Hurley, Illingworth, and Michel \cite[Theorems~1.2 and~4.3]{GHIM}.

There are also consequences for property testing and computation. Hypergraph removal is closely related to property testing \cite{RodlSchacht07}. Let $\cP$ be a graph property and suppose that every graph containing a member of an $F$-abundant family $\cW$ lies outside $\cP$. If there are constants $a,b>0$ such that every graph that is $\eps$-far from $\cP$ is $a\eps^b$-far from being $F$-free, then the member supplied by Theorem~\ref{thmcompact} gives tests whose number of sampled vertices is polynomial in $1/\eps$.  On the other hand, there is no algorithm that decides whether a family of graphs enumerated by a Turing machine is $K_3$-abundant.

\paragraph{Proof sketch.} We prove the nontrivial direction of Theorem~\ref{thmcompact} by contraposition. Suppose that no member of the family is $F$-abundant, fix $c,C>0$, and set $A=C+1$ and $B=\lfloor2rA\rfloor$. For an $F$-coloured family, only finitely many $F$-coloured isomorphism types $(H^\circ,\sigma^\circ)$ with $v(H^\circ)\le B$ occur. Taking products of graphs for which these finitely many graphs fail the abundance inequality gives a uniformly $\lambda$-far $F$-partite $r$-graph $P$ such that $\Hom_{\sigma^\circ}(H^\circ,P)\le\lambda^L v(P)^{v(H^\circ)}$ for every one of these types, where $L>2A$. For a family without colours, we apply the same argument to every $F$-colouring of each $H^\circ$ with at most $B$ vertices.

Let $Q$ be the disjoint union of $m=\lceil\lambda^{-1/(r-1)}\rceil$ copies of $P$ and put $\eps=\lambda/m^{r-1}$. Then $Q$ is uniformly $\eps$-far and $2^{-(r-1)}\lambda^2\le\eps\le\lambda^2$. The preceding inequality applies when $v(H^\circ)\le B$. If $k=v(H^\circ)>B$, join two vertices of $H^\circ$ whenever they belong to a common edge. Every connected component of the resulting graph contains at least $r$ vertices, and counting homomorphisms into the $m$ disjoint copies of $P$ gives $\Hom(H^\circ,Q)\le m^{-(1-1/r)k}v(Q)^k\le\lambda^{k/r}v(Q)^k$. The choice of $B$ and a sufficiently small $\lambda$ imply that both bounds are smaller than $c\eps^C$. When the family has no fixed colouring, every embedding into $Q$ determines an $F$-colouring and there are at most $v(F)^B$ colourings when $v(H^\circ)\le B$. After ignoring its partition, $Q$ is $\eps/v(F)$-far from being $F$-free, and choosing $\lambda$ still smaller makes the resulting bounds less than $c(\eps/v(F))^C$. Thus the family is not $F$-abundant. The same calculation gives Theorem~\ref{thmboundedcore}.

For Corollary~\ref{corfixedcolour}, first suppose that $F$ has no isolated vertices and take all $F$-colourings of an $F$-abundant $r$-graph $H$. This set is nonempty because $H$ must occur in every sufficiently large uniformly far $F$-partite graph after ignoring its partition. Every embedding of $H$ into an $F$-partite graph determines one of these finitely many colourings, so the resulting finite family of $F$-coloured graphs is $F$-abundant. Theorem~\ref{thmcompact} supplies one $F$-abundant colouring. The converse follows by applying Lemma~\ref{lemcanonical} to an $\eps$-far graph and then applying the abundance of $(H,\sigma)$ to the resulting $F$-partite subgraph. If $F$ has isolated vertices, delete them to obtain $F^+$. Every copy of $F^+$ in a sufficiently large graph extends to a copy of $F$, while restricting an $F$-partition to the parts indexed by $V(F^+)$ preserves the required collection of copies. This reduces the assertion to the case with no isolated vertices. The estimates following Corollary~\ref{corGHIM} then show that Theorem~\ref{thmcompact} implies Corollary~\ref{corGHIM} with the family definition of Gir\~ao, Hurley, Illingworth, and Michel.

In the remainder of the paper, we fix an integer $r\ge2$ and an $r$-graph $F$ with at least one edge and no isolated vertices, except where isolated vertices are explicitly allowed. Section~\ref{secproducts} establishes consequences of uniform farness, relates $F$-abundance for graphs with and without a fixed $F$-colouring, and proves the estimates for products and disjoint unions used in Theorem~\ref{thmcompact}. Section~\ref{seccompactproof} proves Theorems~\ref{thmcompact} and~\ref{thmboundedcore} and their corollaries. It also treats isolated vertices and proves Corollary~\ref{corGHIM}. Section~\ref{secapplications} associates translation-invariant linear systems and integer matrices with $F$-coloured $r$-graphs and proves polynomial lower bounds for the numbers of solutions. We also give applications to property testing and prove that no algorithm decides whether a family enumerated by a Turing machine is $K_3$-abundant.

\section{Preliminaries}\label{secproducts}

Throughout this section, $r\ge2$ and $F$ is a fixed $r$-graph with at least one edge and no isolated vertices.
For a finite set $X$, we write $|X|$ for its cardinality. For a nonnegative integer $n$, we write $[n]=\{1,\ldots,n\}$, where $[0]=\varnothing$. For a finite set $X$ and a nonnegative integer $k$, the notation $\binom{X}{k}$ denotes the family of all $k$-element subsets of $X$. If $\phi$ is a map whose domain contains $S$, then $\phi(S)=\{\phi(x)\mid x\in S\}$. An $r$-graph $H$ has vertex set $V(H)$ and edge set $E(H)\subseteq\binom{V(H)}{r}$, and we write $v(H)=|V(H)|$ and $e(H)=|E(H)|$. A vertex is isolated if it belongs to no edge. For $U\subseteq V(H)$, the $r$-graph $H[U]$ has vertex set $U$ and edge set $E(H)\cap\binom{U}{r}$. For $D\subseteq E(H)$, the $r$-graph $H-D$ has vertex set $V(H)$ and edge set $E(H)\setminus D$. If $H_1$ and $H_2$ have disjoint vertex sets, then $H_1\mathbin{\dot\cup}H_2$ denotes their disjoint union. For an integer $x$ and a nonnegative integer $s$, we write $(x)_s=\prod_{i=0}^{s-1}(x-i)$, where $(x)_0=1$. For graphs, $K_n$, $C_n$, and $I_n$ denote the complete graph, the cycle, and the edgeless graph on $n$ vertices, and $\overline H$ denotes the complement of $H$.

For a set $X$, the map $\operatorname{id}_X$ is the identity map on $X$. For a prime power $q$, the field with $q$ elements is denoted by $\mathbb F_q$, and $\mathbb F_q[X]$ denotes the ring of polynomials in $X$ with coefficients in $\mathbb F_q$. The ring of integers modulo $n$ is denoted by $\mathbb Z_n$. If $\mathbb K$ is a field and $X$ is finite, then $\mathbb K^X$ denotes the vector space of maps from $X$ to $\mathbb K$. Let $\mathbf 1$ denote the all-one vector whose index set is determined by the context. All logarithms are natural.

The first lemma shows that uniform farness forces every part of an $F$-partite $r$-graph to contain a positive proportion of its vertices. It also shows that the $r$-graph remains far from being $F$-free after the partition is ignored.

\begin{lemma}\label{lemuniform}
Let $G$ be a nonempty uniformly $\eps$-far $F$-partite $r$-graph with parts $(V_a(G))_{a\in V(F)}$. Then $|V_a(G)|\ge\eps v(G)$ for every $a\in V(F)$. After its partition is ignored, $G$ is $\frac{\eps}{v(F)}$-far from being $F$-free.
\end{lemma}

\begin{proof}
Write $n=v(G)$ and let $\cC$ be a collection from the definition of uniform farness. The family $\cC$ is nonempty. Fix $a\in V(F)$, choose an edge $\eta\in E(F)$ containing $a$, and choose $b\in\eta\setminus\{a\}$. The part $V_b(G)$ contains a vertex $x$ because every member of $\cC$ uses every colour of $F$. At least $\eps n^{r-1}$ members of $\cC$ contain $x$. Their images of $\eta$ are distinct edges because the members of $\cC$ are edge disjoint. At most $|V_a(G)|n^{r-2}$ edges contain $x$ and meet $V_a(G)$, whence $|V_a(G)|\ge\eps n$. Counting incidences between vertices and members of $\cC$ gives $v(F)|\cC|\ge\eps n^r$. Every deletion set that destroys all copies of $F$ meets every member of $\cC$, and the members are edge disjoint. At least $|\cC|\ge\frac{\eps}{v(F)}n^r$ edges must therefore be deleted.
\end{proof}

The next lemma gives a property in the other direction. It finds a large uniformly far $F$-partite subgraph inside every $r$-graph that is far from being $F$-free.

\begin{lemma}\label{lemcanonical}
There is a constant $c_F>0$ with the following property. If an $n$-vertex $r$-graph $G$ is $\eps$-far from being $F$-free, then $G$ has an $F$-partite subgraph $G'$ such that $v(G')\ge c_F\eps^{1/r}n$ and $G'$ is uniformly $c_F\eps$-far from being $F$-free.
\end{lemma}

\begin{proof}
Let $\cM$ be a maximal collection of pairwise edge-disjoint copies of $F$ in $G$. Deleting the edges used by $\cM$ destroys every copy of $F$, so $e(F)|\cM|\ge\eps n^r$. Fix an embedding $\psi\colon F\to G$ for every member of $\cM$, and choose independently and uniformly an element of $V(F)$ for every vertex of $G$. For each fixed $\psi$, the probability that the resulting map $\kappa_G\colon V(G)\to V(F)$ satisfies $\kappa_G\circ\psi=\operatorname{id}_{V(F)}$ is $v(F)^{-v(F)}$. Some choice of $\kappa_G$ therefore retains a subcollection $\cC_1$ with $|\cC_1|\ge\frac{\eps}{e(F)v(F)^{v(F)}}n^r$.

Set $\alpha=\frac{\eps}{2e(F)v(F)^{v(F)}}$. Starting with $\cC_1$, whenever a vertex belongs to at least one and fewer than $\alpha n^{r-1}$ remaining copies, delete all remaining copies through that vertex and never select the vertex again. Fewer than $\alpha n^r$ copies are deleted because there are at most $n$ selections. The final collection $\cC'$ satisfies $|\cC'|\ge\frac{\eps}{2e(F)v(F)^{v(F)}}n^r$. Let $G'$ consist of the vertices and edges used by $\cC'$. Every vertex of $G'$ belongs to at least $\alpha n^{r-1}\ge\alpha v(G')^{r-1}$ members of $\cC'$, so $G'$ is uniformly $\alpha$-far. Fix one edge of $F$. Its images in the members of $\cC'$ are distinct, and hence $v(G')^r\ge|\cC'|$. It follows that $v(G')\ge(2e(F)v(F)^{v(F)})^{-1/r}\eps^{1/r}n$. The assertion follows with $c_F=\min\{(2e(F)v(F)^{v(F)})^{-1},(2e(F)v(F)^{v(F)})^{-1/r}\}$.
\end{proof}

The preceding reduction allows a lower bound counted by $\Emb_\sigma$ to be used without prescribing an $F$-colouring. The following lemma states the resulting implication using the definitions from Section \ref{sec:intro}.

\begin{lemma}\label{lemcolourtouncolour}
If the $F$-coloured $r$-graph $(H,\sigma)$ is $F$-abundant, then the $r$-graph $H$ is $F$-abundant.
\end{lemma}

\begin{proof}
Let $h=v(H)$ and let $c,C>0$ be constants for which the definition of $F$-abundance holds for $(H,\sigma)$. Apply Lemma~\ref{lemcanonical} to an $\eps$-far $n$-vertex $r$-graph $G$. Once $n$ is sufficiently large in terms of $\eps$, the subgraph $G'$ supplied by that lemma is large enough to apply the definition of $F$-abundance to $(H,\sigma)$. Every embedding counted by $\Emb_\sigma(H,G')$ is counted by $\Emb(H,G)$, and therefore
\[
\Emb(H,G)\ge c(c_F\eps)^C(c_F\eps^{1/r}n)^h=cc_F^{C+h}\eps^{C+h/r}n^h.
\]
Thus $H$ is $F$-abundant.
\end{proof}

The next lemma permits two changes in the estimates used below: Embeddings may be replaced by homomorphisms, and an $F$-abundance estimate remains valid after isolated vertices are added to the $r$-graph being counted.

\begin{lemma}\label{lembasiccounts}
Let $(H,\sigma)$ be a fixed $F$-coloured $r$-graph. The pair $(H,\sigma)$ is $F$-abundant if and only if there are constants $c,C>0$ such that, for every $0<\eps\le1$, there is $n_0(\eps)$ for which every uniformly $\eps$-far $n$-vertex $F$-partite $r$-graph $G$ with $n\ge n_0(\eps)$ satisfies $\Hom_\sigma(H,G)\ge c\eps^Cn^{v(H)}$. If $(H^\circ,\sigma^\circ)$ is $F$-abundant, then $(H,\sigma)$ is $F$-abundant. The analogous statements hold for an $r$-graph $H$ with $\Emb(H,G)$ and $\Hom(H,G)$.
\end{lemma}

\begin{proof}
Write $h=v(H)$. The number of non-injective maps from $V(H)$ to an $n$-element set is at most $\binom{h}{2}n^{h-1}$. A homomorphism lower bound $c\eps^C n^h$ therefore gives an embedding lower bound $\frac{c}{2}\eps^C n^h$ once $n\ge\frac{2\binom{h}{2}}{c\eps^C}$. The reverse implication follows from $\Hom_\sigma(H,G)\ge\Emb_\sigma(H,G)$.

Suppose that $H$ has $s$ isolated vertices and that $(H^\circ,\sigma^\circ)$ satisfies an $F$-abundance estimate $c\eps^C n^{h-s}$. Lemma~\ref{lemuniform} gives at least $\eps n$ vertices in every part. For $n\ge\frac{2h}{\eps}$, every embedding counted by $\Emb_{\sigma^\circ}(H^\circ,G)$ has at least $(\frac{\eps n}{2})^s$ extensions counted by $\Emb_\sigma(H,G)$. Thus $\Emb_\sigma(H,G)\ge c2^{-s}\eps^{C+s}n^h$. For an $r$-graph without a fixed $F$-colouring, each isolated vertex has at least $n-h$ choices. This proves the corresponding statement for $H$.
\end{proof}

The next lemma gives the consequence of the failure of $F$-abundance needed for the product argument. For every prescribed exponent, it produces an arbitrarily small value of $\eta$ for which arbitrarily large uniformly $\eta$-far $F$-partite $r$-graphs contain few homomorphisms from the fixed $F$-coloured $r$-graph.

\begin{lemma}\label{lembad}
Suppose that a fixed $F$-coloured $r$-graph $(H,\sigma)$ is not $F$-abundant. For every $D>0$ and $\eta_0>0$, there is $0<\eta<\eta_0$ such that, for every $N_0$, some uniformly $\eta$-far $F$-partite $r$-graph $G$ satisfies $v(G)\ge N_0$ and $\Hom_\sigma(H,G)<\eta^D v(G)^{v(H)}$.
\end{lemma}

\begin{proof}
Suppose the assertion fails for some $D$ and $\eta_0$. For every $0<\eta<\eta_0$ there is then $N(\eta)$ such that every uniformly $\eta$-far $G$ above this threshold satisfies $\Hom_\sigma(H,G)\ge\eta^D v(G)^{v(H)}$. Set $\eta_*=\min\{\frac{\eta_0}{2},1\}$. If $0<\eps<\eta_*$, use the estimate at $\eta=\eps$. If $\eta_*\le\eps\le1$, a collection that shows uniform $\eps$-farness also shows uniform $\eta_*$-farness, and the estimate at $\eta_*$ gives $\Hom_\sigma(H,G)\ge\eta_*^D\eps^D v(G)^{v(H)}$. Lemma~\ref{lembasiccounts} now implies that $(H,\sigma)$ is $F$-abundant, a contradiction.
\end{proof}

Let $\ell\ge1$, and let $G_1,\ldots,G_\ell$ be $F$-partite $r$-graphs. Their product $P$ has parts $V_a(P)=\prod_{j\in[\ell]}V_a(G_j)$ for $a\in V(F)$. For $\eta\in E(F)$ and vertices $x_a=(x_{a,1},\ldots,x_{a,\ell})\in V_a(P)$ with $a\in\eta$, the set $\{x_a\mid a\in\eta\}$ is an edge of $P$ if and only if $\{x_{a,j}\mid a\in\eta\}$ is an edge of $G_j$ for every $j\in[\ell]$. This defines an $F$-partite $r$-graph.

The following lemma combines the preceding examples for finitely many $F$-coloured $r$-graphs, none of which is $F$-abundant. It produces one uniformly far $F$-partite $r$-graph in which all of them have simultaneously few homomorphisms.

\begin{lemma}\label{lemtensor}
Let $\cS=\{(H_1,\sigma_1),\ldots,(H_s,\sigma_s)\}$ be a finite collection of $F$-coloured $r$-graphs without isolated vertices, none of which is $F$-abundant. For every $L>0$ and $\lambda_*>0$, there is $0<\lambda<\lambda_*$ such that, for every $N_0$, some uniformly $\lambda$-far $F$-partite $r$-graph $P$ satisfies $v(P)\ge N_0$ and $\Hom_{\sigma_i}(H_i,P)\le\lambda^L v(P)^{v(H_i)}$ for every $i\in\{1,\ldots,s\}$.
\end{lemma}

\begin{proof}
We first handle $s=0$. Choose $0<\lambda<\min\{\lambda_*,v(F)^{-(r-1)}\}$. Let $q>v(F)$ be an arbitrarily large prime, choose distinct $\rho(a)\in\mathbb{F}_q$ for $a\in V(F)$, and take the parts $V_a(P)=\{a\}\times\mathbb{F}_q$. For every polynomial $p\in\mathbb{F}_q[X]$ of degree less than $r$, define $\psi_p\colon V(F)\to V(P)$ by $\psi_p(a)=(a,p(\rho(a)))$, and put $\psi_p(\eta)$ in $E(P)$ for every $\eta\in E(F)$. If $\psi_p(F)$ and $\psi_{p'}(F)$ share an edge, then $p$ and $p'$ agree at the $r$ distinct values indexed by that edge, so $p=p'$. The sets $\psi_p(E(F))$ are therefore pairwise disjoint. Every vertex belongs to exactly $q^{r-1}$ of these copies, while $v(P)=v(F)q$. Thus $P$ is uniformly $v(F)^{-(r-1)}$-far and hence uniformly $\lambda$-far.

Assume that $s\ge1$, write $h_i=v(H_i)$ and $h=\max\{h_i\mid1\le i\le s\}$, and choose $D>2s(L+h+1)$. Lemma~\ref{lembad} gives numbers $0<\eta_i<\frac{1}{2}$ such that arbitrarily large uniformly $\eta_i$-far $G_i$ satisfy $\Hom_{\sigma_i}(H_i,G_i)<\eta_i^D v(G_i)^{h_i}$. Let $\ell_i=\log(1/\eta_i)$. Choose $T$ so large that $T\ge\ell_i$ for every $i$, $sT\ge(r-1)L\log v(F)$, and $e^{-sT}v(F)^{-(r-1)}<\lambda_*$. Define
\[
k_i=\left\lceil\frac{T}{\ell_i}\right\rceil,\qquad w_i=k_i\ell_i,\qquad W=\sum_{i=1}^s w_i,\qquad \lambda_0=e^{-W},\qquad \lambda=\frac{\lambda_0}{v(F)^{r-1}}.
\]
Then $T\le w_i\le2T$, $w_i\ge\frac{W}{2s}$, and $0<\lambda<\lambda_*$. Choose the $G_i$ sufficiently large and take the product with $k_i$ coordinates isomorphic to $G_i$. Write $n_i=v(G_i)$ and $M=\prod_{i=1}^s n_i^{k_i}$. Its part of colour $a$ is $V_a(P)=\prod_{i=1}^s V_a(G_i)^{k_i}$.

Take the coordinatewise products of the edge-disjoint collections that show uniform farness in the factors. Two product copies sharing an edge project in every coordinate to copies sharing an edge. The collections in the factors are edge disjoint, so all coordinate choices agree and the two product copies are equal. Every vertex belongs to at least $\lambda_0M^{r-1}$ product copies. Since $v(P)\le v(F)M$, this is at least $\lambda v(P)^{r-1}$. Lemma~\ref{lemuniform} also gives $|V_a(P)|\ge\lambda_0M$ for every $a$, whence $v(P)\ge v(F)\lambda_0M$ and $M\le\lambda_0^{-1}v(P)$. The factors can be chosen large enough to ensure $v(P)\ge N_0$.

Fix $j\in\{1,\ldots,s\}$. A homomorphism to the product projects to a homomorphism in every coordinate. The $k_j$ coordinates isomorphic to $G_j$ give the required upper bound, while every other coordinate has at most $n_i^{h_j}$ maps. Consequently
\[
\begin{aligned}
\Hom_{\sigma_j}(H_j,P)&\le\eta_j^{Dk_j}M^{h_j}\le\exp(-Dw_j+h_jW)v(P)^{h_j}\\
&\le\exp(-(L+1)W)v(P)^{h_j}\le\lambda^L v(P)^{h_j}.
\end{aligned}
\]
The penultimate inequality follows from $w_j\ge\frac{W}{2s}$ and the choice of $D$. The last follows from $W\ge(r-1)L\log v(F)$. This proves the lemma.
\end{proof}

Two vertices $x$ and $y$ of an $r$-graph $J$ are in the same component if $x=y$ or there are edges $e_1,\ldots,e_d\in E(J)$ such that $x\in e_1$, $y\in e_d$, and $e_i\cap e_{i+1}\ne\varnothing$ for every $i\in[d-1]$. The components of $J$ are the equivalence classes defined by this relation.

The preceding product controls finitely many fixed $r$-graphs, while the family in Theorem~\ref{thmcompact} may contain graphs with arbitrarily many vertices. The next lemma shows that taking many disjoint copies preserves uniform farness with a weaker parameter and bounds the number of homomorphisms from every $r$-graph without isolated vertices in terms of its number of vertices.

\begin{lemma}\label{lemdisjoint}
Let $Q$ be the disjoint union of $m$ copies of a uniformly $\lambda$-far $F$-partite $r$-graph $P$. Then $Q$ is uniformly $\frac{\lambda}{m^{r-1}}$-far. If $K$ has no isolated vertices and $k=v(K)$, then $\Hom(K,Q)\le m^{-(1-1/r)k}v(Q)^k$.
\end{lemma}

\begin{proof}
Every vertex lies in one copy of $P$ and belongs there to at least $\lambda v(P)^{r-1}=\frac{\lambda}{m^{r-1}}v(Q)^{r-1}$ copies from a collection showing uniform farness. This proves the first assertion. Let $a$ be the number of components of $K$. Every component contains an edge and hence has at least $r$ vertices, so $a\le\frac{k}{r}$. The image of each component under a homomorphism from $K$ to $Q$ lies in one of the $m$ copies of $P$. Choosing that copy for every component and then choosing the images of the vertices gives
\[
\Hom(K,Q)\le m^av(P)^k=m^{a-k}v(Q)^k\le m^{-(1-1/r)k}v(Q)^k.
\]
\end{proof}

\section{Proof of compactness}\label{seccompactproof}


We first prove Theorem~\ref{thmcompact}. The implication from an $F$-abundant member to an $F$-abundant family follows directly from the definitions, while the converse is proved by contraposition. For arbitrary constants $c,C>0$, we construct arbitrarily large $r$-graphs in which every member of the family fails the lower bound required by $c$ and $C$. We first consider families of $F$-coloured $r$-graphs and then consider families of $r$-graphs.

\begin{proof}[Proof of Theorem~\ref{thmcompact}]
If the family contains an $F$-abundant member, then the constants for that member also satisfy the definition of $F$-abundance for the family. We prove the converse by contraposition. Suppose that no member of a family $\cH$ of $F$-coloured $r$-graphs is $F$-abundant. Fix arbitrary constants $c,C>0$. Let $A=C+1$, $L=2A+1$, and $B=\lfloor2rA\rfloor$. Since $L-2A>0$, $\frac{B+1}{r}-2A>0$, and $A-C>0$, we may choose $0<\lambda_*\le1$ such that
\[
\lambda_*^{L-2A}<2^{-(r-1)A},\qquad \lambda_*^{\frac{B+1}{r}-2A}<2^{-(r-1)A},\qquad \lambda_*^{2(A-C)}<c.
\]

For every $(H,\sigma)\in\cH$, Lemma~\ref{lembasiccounts} shows that $(H^\circ,\sigma^\circ)$ is not $F$-abundant. Moreover, $H^\circ$ is nonempty. To see this, the $F$-coloured $r$-graph on $[0]$ has exactly one embedding into every $F$-partite $r$-graph and is therefore $F$-abundant. Lemma~\ref{lembasiccounts} would then imply that every $F$-coloured $r$-graph consisting only of isolated vertices is $F$-abundant, contrary to the assumption on $\cH$.

For each integer $k\in\{1,\ldots,B\}$, an $r$-graph on $[k]$ has at most $2^{\binom{k}{r}}$ possible edge sets, and there are at most $v(F)^k$ maps from $[k]$ to $V(F)$. After relabelling vertices, the pairs $(H^\circ,\sigma^\circ)$ arising from members of $\cH$ with $v(H^\circ)\le B$ therefore give only finitely many pairs. Apply Lemma~\ref{lemtensor} to one copy of every such pair. It gives $0<\lambda<\lambda_*$ and, above every prescribed order, a uniformly $\lambda$-far $F$-partite $r$-graph $P$ satisfying $\Hom_{\sigma^\circ}(H^\circ,P)\le\lambda^L v(P)^{v(H^\circ)}$ for every $(H,\sigma)\in\cH$ with $v(H^\circ)\le B$.

Let $m=\lceil\lambda^{-\frac{1}{r-1}}\rceil$, let $Q$ be the disjoint union of $m$ copies of $P$, and let $\eps=\frac{\lambda}{m^{r-1}}$. Since $0<\lambda\le1$, we have $\lambda^{-\frac{1}{r-1}}\le m\le2\lambda^{-\frac{1}{r-1}}$. It follows that $2^{-(r-1)}\lambda^2\le\eps\le\lambda^2$. Lemma~\ref{lemdisjoint} shows that $Q$ is uniformly $\eps$-far. Once $\lambda$ is fixed, so is $m$, while $P$ may have arbitrarily large order. Hence $Q$ may also have arbitrarily large order.

Fix $(H,\sigma)\in\cH$, let $K=H^\circ$, let $k=v(K)$, and let $a$ be the number of components of $K$. Suppose first that $k\le B$. Every component of $K$ contains an edge. Under a homomorphism from $K$ to $Q$, the image of this edge lies in one copy of $P$, and the definition of a component forces the image of the whole component to lie in that same copy. Each component may choose any one of the $m$ copies of $P$, independently of the other components. Consequently $\Hom_{\sigma^\circ}(K,Q)=m^a\Hom_{\sigma^\circ}(K,P)$. The isolated vertices of $H$ have at most $v(Q)^{v(H)-k}$ possible images, and $v(Q)=mv(P)$. Therefore
\[
\Emb_\sigma(H,Q)\le v(Q)^{v(H)-k}m^a\Hom_{\sigma^\circ}(K,P)\le m^{a-k}\lambda^L v(Q)^{v(H)}\le\lambda^L v(Q)^{v(H)}.
\]
The last inequality uses $a\le\frac{k}{r}$, which follows because every component of $K$ contains at least $r$ vertices. The lower bound on $\eps$ and the choice of $\lambda_*$ give $\eps^A\ge2^{-(r-1)A}\lambda^{2A}>\lambda^L$. Hence $\Emb_\sigma(H,Q)<\eps^A v(Q)^{v(H)}$.

Suppose next that $k>B$. Restricting an embedding of $H$ to $K$ and then forgetting the $F$-colouring gives
\[
\Emb_\sigma(H,Q)\le v(Q)^{v(H)-k}\Hom(K,Q)\le m^{-(1-\frac{1}{r})k}v(Q)^{v(H)}\le\lambda^{\frac{k}{r}}v(Q)^{v(H)}.
\]
The last inequality follows from $m\ge\lambda^{-\frac{1}{r-1}}$. Since $k\ge B+1$ and $0<\lambda\le1$, we have $\lambda^{\frac{k}{r}}\le\lambda^{\frac{B+1}{r}}$. The lower bound on $\eps$ and the second condition on $\lambda_*$ give $\lambda^{\frac{B+1}{r}}<\eps^A$. Thus $\Emb_\sigma(H,Q)<\eps^A v(Q)^{v(H)}$ also in this case.

We have proved the same inequality for every member of $\cH$. Since $\eps\le\lambda^2$ and $A-C>0$, the third condition on $\lambda_*$ gives $\eps^{A-C}\le\lambda^{2(A-C)}<c$. It follows that $\Emb_\sigma(H,Q)<c\eps^C v(Q)^{v(H)}$ for every $(H,\sigma)\in\cH$. For the fixed value of $\eps$, such an $r$-graph $Q$ exists above every prescribed order. The constants $c$ and $C$ therefore do not satisfy the definition of $F$-abundance for $\cH$. Since they were arbitrary, the family $\cH$ is not $F$-abundant.

We next consider families of $r$-graphs. We apply Lemma~\ref{lemtensor} simultaneously to every $F$-colouring of each graph $H^\circ$ with at most $B$ vertices, and we count separately over the possible $F$-colourings induced by embeddings into the final $F$-partite $r$-graph. As before, a family containing an $F$-abundant member is $F$-abundant.

 We prove the converse by contraposition. Suppose that no member of a family $\cU$ of $r$-graphs is $F$-abundant. Fix arbitrary constants $c,C>0$, and let $A=C+1$, $L=2A+1$, and $B=\lfloor2rA\rfloor$. Choose $0<\lambda_*\le1$ such that
\[
\lambda_*^{L-2A}<2^{-(r-1)A}v(F)^{-B},\qquad \lambda_*^{\frac{B+1}{r}-2A}<2^{-(r-1)A},\qquad \lambda_*^{2(A-C)}<cv(F)^{-C}.
\]

For each $H\in\cU$, let $K=H^\circ$. The $r$-graph $K$ is nonempty because the $r$-graph on $[0]$ is $F$-abundant and Lemma~\ref{lembasiccounts} adds isolated vertices. For every $F$-colouring $\tau$ of $K$, the $F$-coloured $r$-graph $(K,\tau)$ is not $F$-abundant. Otherwise Lemma~\ref{lemcolourtouncolour} would show that the $r$-graph $K$ is $F$-abundant, and Lemma~\ref{lembasiccounts} would then show that $H$ is $F$-abundant.

After relabelling vertices as in the preceding proof, the pairs $(K,\tau)$ with $H\in\cU$ and $v(K)\le B$ form a finite collection. Lemma~\ref{lemtensor} gives $0<\lambda<\lambda_*$ and arbitrarily large uniformly $\lambda$-far $F$-partite $r$-graphs $P$ satisfying $\Hom_\tau(K,P)\le\lambda^L v(P)^{v(K)}$ for every pair in this collection. Let $m=\lceil\lambda^{-\frac{1}{r-1}}\rceil$, let $Q$ be the disjoint union of $m$ copies of $P$, and let $\eps=\frac{\lambda}{m^{r-1}}$. Lemma~\ref{lemdisjoint} shows that $Q$ is uniformly $\eps$-far, and Lemma~\ref{lemuniform} shows that the $r$-graph obtained after its $F$-partition is forgotten is $\delta$-far from being $F$-free, where $\delta=\frac{\eps}{v(F)}$. The definition of $m$ again gives $2^{-(r-1)}\lambda^2\le\eps\le\lambda^2$.

Fix $H\in\cU$, let $K=H^\circ$, let $k=v(K)$, and let $a$ be the number of components of $K$. Suppose first that $k\le B$. Every embedding $\phi$ of $K$ into $Q$ determines the $F$-colouring $\kappa_Q\circ\phi$ of $K$. There are at most $v(F)^k$ such colourings. For each fixed $F$-colouring $\tau$, the component argument from the preceding proof gives $\Hom_\tau(K,Q)=m^a\Hom_\tau(K,P)$. Hence
\[
\Emb(H,Q)\le v(Q)^{v(H)-k}\sum_\tau\Hom_\tau(K,Q)\le v(F)^k m^{a-k}\lambda^L v(Q)^{v(H)}\le v(F)^B\lambda^L v(Q)^{v(H)}.
\]
The sum ranges over the $F$-colourings of $K$. If there is no such colouring, the sum is zero. The lower bound on $\eps$ and the first condition on $\lambda_*$ give $v(F)^B\lambda^L<\eps^A$. Thus $\Emb(H,Q)<\eps^A v(Q)^{v(H)}$.

If $k>B$, restricting every embedding of $H$ to $K$ and allowing arbitrary images for the isolated vertices gives $\Emb(H,Q)\le v(Q)^{v(H)-k}\Hom(K,Q)$. Lemma~\ref{lemdisjoint} and $m\ge\lambda^{-\frac{1}{r-1}}$ therefore give $\Emb(H,Q)\le\lambda^{\frac{k}{r}}v(Q)^{v(H)}\le\lambda^{\frac{B+1}{r}}v(Q)^{v(H)}<\eps^A v(Q)^{v(H)}$. We have now obtained this last bound for every $H\in\cU$.

Since $\eps\le\lambda^2$, the third condition on $\lambda_*$ gives $\eps^{A-C}<cv(F)^{-C}$. Consequently $\eps^A<c\left(\frac{\eps}{v(F)}\right)^C=c\delta^C$. For this fixed $\delta$, the order of $Q$ may be arbitrarily large, and every member $H\in\cU$ satisfies $\Emb(H,Q)<c\delta^C v(Q)^{v(H)}$. Thus $c$ and $C$ do not satisfy the definition of $F$-abundance for $\cU$. Since $c$ and $C$ were arbitrary, $\cU$ is not $F$-abundant.
\end{proof}

Next we show that an $F$-abundant $r$-graph has one fixed $F$-colouring that is $F$-abundant.
 We first handle an $r$-graph $F$ without isolated vertices and then reduce the general case to $F^\circ$.

\begin{proof}[Proof of Corollary~\ref{corfixedcolour}]
Suppose first that $F$ has no isolated vertices. If an $F$-colouring $\sigma$ of $H$ makes $(H,\sigma)$ $F$-abundant, then Lemma~\ref{lemcolourtouncolour} shows that the $r$-graph $H$ is $F$-abundant. We prove the converse. Suppose that $H$ is $F$-abundant with constants $c,C>0$, and let $\Sigma$ be the finite set of all $F$-colourings of $H$.

The set $\Sigma$ is nonempty. To see this, apply the case $s=0$ of Lemma~\ref{lemtensor} with $L=1$ and $\lambda_*=1$. It gives a fixed $\lambda>0$ and arbitrarily large uniformly $\lambda$-far $F$-partite $r$-graphs $P$. If $\Sigma$ were empty, then $P$ would contain no embedding of $H$, because composing such an embedding with $\kappa_P$ would give an element of $\Sigma$. On the other hand, Lemma~\ref{lemuniform} shows that the $r$-graph obtained from $P$ after its partition is forgotten is $\frac{\lambda}{v(F)}$-far from being $F$-free. Taking $P$ sufficiently large and applying the $F$-abundance of $H$ gives an embedding of $H$ into $P$, a contradiction.

Let $G$ be a uniformly $\eps$-far $F$-partite $r$-graph. After its $F$-partition is forgotten, Lemma~\ref{lemuniform} makes $G$ $\frac{\eps}{v(F)}$-far from being $F$-free. For all sufficiently large $v(G)$, the $F$-abundance of $H$ gives $\Emb(H,G)\ge cv(F)^{-C}\eps^C v(G)^{v(H)}$. Every embedding $\phi$ of $H$ into $G$ determines exactly one member $\kappa_G\circ\phi$ of $\Sigma$, and therefore $\Emb(H,G)=\sum_{\sigma\in\Sigma}\Emb_\sigma(H,G)$. Some $\sigma\in\Sigma$ satisfies $\Emb_\sigma(H,G)\ge\frac{c}{|\Sigma|v(F)^C}\eps^C v(G)^{v(H)}$. Hence the finite family $\{(H,\sigma)\mid\sigma\in\Sigma\}$ is $F$-abundant. The assertion of Theorem~\ref{thmcompact} for families of $F$-coloured $r$-graphs supplies one $\sigma\in\Sigma$ for which $(H,\sigma)$ is $F$-abundant.

Now allow $F$ to have isolated vertices. Since $e(F)>0$, the $r$-graph $F^\circ$ is nonempty and has no isolated vertices. If an $r$-graph has at least $v(F)$ vertices, every embedding of $F^\circ$ extends to an embedding of $F$ by choosing distinct images for the isolated vertices. Consequently a set of edges destroys every copy of $F$ if and only if it destroys every copy of $F^\circ$. Thus the definition of $F$-abundance for an $r$-graph is equivalent to the definition of $F^\circ$-abundance after the order threshold is increased if necessary.

Suppose that the $r$-graph $H$ is $F$-abundant. By the case already proved, there is an $F^\circ$-colouring $\sigma$ of $H$ such that $(H,\sigma)$ is $F^\circ$-abundant. Regard $\sigma$ as an $F$-colouring through the inclusion of $F^\circ$ in $F$. Let $c,D>0$ satisfy the definition of $F^\circ$-abundance for $(H,\sigma)$, and let $G$ be a uniformly $\eps$-far $F$-partite $r$-graph with $n=v(G)$. Let $U=\bigcup_{a\in V(F^\circ)}V_a(G)$ and let $G'=G[U]$, with the $F^\circ$-partition inherited from $G$.

Let $\cC$ be a collection from the definition of uniform $\eps$-farness for $G$. Restrict every $\psi\in\cC$ to $F^\circ$. The resulting embeddings have pairwise disjoint edge sets because $E(F^\circ)=E(F)$, and every vertex of $G'$ belongs to at least $\eps n^{r-1}\ge\eps v(G')^{r-1}$ of them. Hence $G'$ is uniformly $\eps$-far from being $F^\circ$-free. Counting incidences between the vertices of $G$ and the members of $\cC$ gives $v(F)|\cC|\ge\eps n^r$. Fix an edge $\eta\in E(F^\circ)$. The images $\psi(\eta)$ are distinct for different $\psi\in\cC$ and are edges of $G'$. It follows that $v(G')^r\ge|\cC|\ge\frac{\eps}{v(F)}n^r$, and therefore $v(G')\ge\left(\frac{\eps}{v(F)}\right)^{\frac{1}{r}}n$.

For all sufficiently large $n$ in terms of $\eps$, the $r$-graph $G'$ is large enough to apply the $F^\circ$-abundance of $(H,\sigma)$. We obtain $\Emb_\sigma(H,G)\ge\Emb_\sigma(H,G')\ge cv(F)^{-\frac{v(H)}{r}}\eps^{D+\frac{v(H)}{r}}n^{v(H)}$. Thus $(H,\sigma)$ is $F$-abundant.

It remains to prove the converse when $F$ has isolated vertices. The proof of Lemma~\ref{lemcanonical} uses the assumption $e(F)>0$ but does not use the absence of isolated vertices. The maximal edge-disjoint collection, the random map to $V(F)$, and the deletion process are unchanged. The lower bound for the number of vertices uses the images of one fixed edge of $F$, which exists because $e(F)>0$. Hence there is a constant $c_F>0$ such that every $\eps$-far $n$-vertex $r$-graph has an $F$-partite subgraph $G'$ satisfying $v(G')\ge c_F\eps^{\frac{1}{r}}n$ and uniformly $c_F\eps$-far from being $F$-free. If $(H,\sigma)$ is $F$-abundant with constants $c,D>0$, applying its lower bound inside this $G'$ gives $\Emb(H,G)\ge cc_F^{D+v(H)}\eps^{D+\frac{v(H)}{r}}n^{v(H)}$ for every sufficiently large $n$. Therefore the $r$-graph $H$ is $F$-abundant.
\end{proof}

It remains to compare the definitions in this paper with the definitions preceding Question~\ref{quesGHIM} in \cite[Question~5.2]{GHIM}. We treat families of graphs and families of $K_3$-coloured graphs separately, and in each case we write both directions of the comparison.

\begin{proof}[Proof of Corollary~\ref{corGHIM}]
Conjecture~\ref{conjGHIM} follows from Corollary~\ref{corfixedcolour} with $r=2$. We prove the two assertions in Question~\ref{quesGHIM}. The copy counts $N(H,G)$ and $N_\sigma(H,G)$ are those defined in the Introduction.

First let $\cA$ be a family of graphs that is triangle-abundant in the sense of Gir\~ao, Hurley, Illingworth, and Michel. Let $G$ be an $n$-vertex graph that is $\eps$-far from being triangle-free, and let $\cM$ be a maximal collection of pairwise edge-disjoint triangles in $G$. Deleting all edges used by $\cM$ makes $G$ triangle-free, since any remaining triangle could be added to $\cM$. Therefore $3|\cM|\ge\eps n^2$. Applying the triangle-abundance of $\cA$ with parameter $\frac{\eps}{3}$ gives some $H\in\cA$ satisfying $N(H,G)\ge\left(\frac{\eps}{3}\right)^C n^{v(H)}$ for a constant $C>0$. Since $\Emb(H,G)=|\operatorname{Aut}(H)|N(H,G)\ge N(H,G)$, the family $\cA$ is $K_3$-abundant under the definition used in Theorem~\ref{thmcompact}. That theorem supplies a graph $H\in\cA$ that is $K_3$-abundant.

Fix constants $c,D>0$ for this graph $H$, and suppose that an $n$-vertex graph $G$ contains at least $\eps n^2$ pairwise edge-disjoint triangles. Every set of edges that makes $G$ triangle-free must meet each of these triangles, so $G$ is $\eps$-far from being triangle-free. For all sufficiently large $n$, we have $N(H,G)\ge\frac{c}{|\operatorname{Aut}(H)|}\eps^D n^{v(H)}$. Whenever such a collection of triangles is nonempty, $3\eps n^2\le e(G)<\frac{n^2}{2}$, so $0<\eps<\frac{1}{6}$. Choose $s\ge0$ such that $6^{-s}\le\frac{c}{|\operatorname{Aut}(H)|}$. Then $\frac{c}{|\operatorname{Aut}(H)|}\eps^D\ge\eps^{D+s}$. Thus the singleton family $\{H\}$ is triangle-abundant in the sense of Gir\~ao, Hurley, Illingworth, and Michel. Conversely, if a family contains a graph whose singleton family is triangle-abundant, then the original family is triangle-abundant by using that graph in every host.

Now let $\cA$ be a family of $K_3$-coloured graphs that is triangle-abundant in the sense of Gir\~ao, Hurley, Illingworth, and Michel. Let $G$ be a uniformly $\eps$-far $K_3$-partite graph, and let $\cC$ be a collection from the definition of uniform farness. Counting incidences between vertices and members of $\cC$ gives $3|\cC|\ge\eps n^2$, where $n=v(G)$. Hence $G$ contains at least $\frac{\eps}{3}n^2$ pairwise edge-disjoint triangles whose three vertices lie in the three prescribed parts. The triangle-abundance of $\cA$ gives some $(H,\sigma)\in\cA$ satisfying $N_\sigma(H,G)\ge\left(\frac{\eps}{3}\right)^C n^{v(H)}$ for a fixed $C>0$. Since $\Emb_\sigma(H,G)=|\operatorname{Aut}(H,\sigma)|N_\sigma(H,G)\ge N_\sigma(H,G)$, the family $\cA$ is $K_3$-abundant under the definition used in Theorem~\ref{thmcompact}. That theorem supplies a $K_3$-coloured graph $(H,\sigma)\in\cA$ that is $K_3$-abundant.

Fix constants $c,D>0$ for $(H,\sigma)$. Suppose that a $K_3$-partite $n$-vertex graph $G$ contains a collection $\cC$ of at least $\eps n^2$ pairwise edge-disjoint triangles whose three vertices lie in the three prescribed parts. Starting with $\cC$, repeatedly choose a vertex that belongs to at least one and fewer than $\frac{\eps}{2}n$ remaining triangles, delete all remaining triangles containing that vertex, and do not choose that vertex again. At most $n$ vertices are chosen, so fewer than $\frac{\eps}{2}n^2$ triangles are deleted. The final collection $\cC'$ has at least $\frac{\eps}{2}n^2$ members. Let $G'$ be the subgraph formed by the vertices and edges used by $\cC'$. Every vertex of $G'$ belongs to at least $\frac{\eps}{2}n\ge\frac{\eps}{2}v(G')$ members of $\cC'$, so $G'$ is uniformly $\frac{\eps}{2}$-far. Fix one edge of $K_3$. Its images in the members of $\cC'$ are distinct, and hence $v(G')^2\ge|\cC'|\ge\frac{\eps}{2}n^2$. Therefore $v(G')\ge\left(\frac{\eps}{2}\right)^{\frac{1}{2}}n$.

For all sufficiently large $n$ in terms of $\eps$, the graph $G'$ is large enough to apply the $K_3$-abundance of $(H,\sigma)$. We obtain $\Emb_\sigma(H,G)\ge c2^{-D-\frac{v(H)}{2}}\eps^{D+\frac{v(H)}{2}}n^{v(H)}$, and hence $N_\sigma(H,G)\ge\frac{c2^{-D-\frac{v(H)}{2}}}{|\operatorname{Aut}(H,\sigma)|}\eps^{D+\frac{v(H)}{2}}n^{v(H)}$. Let $s_\sigma\ge0$ satisfy $6^{-s_\sigma}\le\frac{c2^{-D-\frac{v(H)}{2}}}{|\operatorname{Aut}(H,\sigma)|}$. Since $0<\eps<\frac{1}{6}$ whenever the given collection is nonempty, we have $N_\sigma(H,G)\ge\eps^{D+\frac{v(H)}{2}+s_\sigma}n^{v(H)}$. Thus the singleton family $\{(H,\sigma)\}$ is triangle-abundant in the sense of Gir\~ao, Hurley, Illingworth, and Michel. The reverse implication follows directly by using a triangle-abundant member of the family in every graph in the definition.
\end{proof}

We now prove Theorem~\ref{thmboundedcore}. We first consider families of $F$-coloured $r$-graphs and then families of $r$-graphs. In both cases, the constants $c$ and $C$ are fixed by the $F$-abundance of the family, and we construct a sufficiently large graph that contradicts the lower bound associated with these constants.

\begin{proof}[Proof of Theorem~\ref{thmboundedcore}]
Let $\cH$ be $F$-abundant with constants $c,C>0$, let $A=C+1$, let $L=2A+1$, and let $B=\lfloor2rA\rfloor$. Suppose for a contradiction that no $(H,\sigma)\in\cH$ is both $F$-abundant and satisfies $v(H^\circ)\le B$. Choose $0<\lambda_*\le1$ such that
\[
\lambda_*^{L-2A}<2^{-(r-1)A},\qquad \lambda_*^{\frac{B+1}{r}-2A}<2^{-(r-1)A},\qquad \lambda_*^{2(A-C)}<c.
\]

If $(H,\sigma)\in\cH$ and $v(H^\circ)\le B$, then $(H^\circ,\sigma^\circ)$ is not $F$-abundant. Otherwise Lemma~\ref{lembasiccounts} would show that $(H,\sigma)$ is $F$-abundant, contrary to the assumption. Apply Lemma~\ref{lemtensor} to the finitely many relabelled pairs $(H^\circ,\sigma^\circ)$ that arise in this way. We obtain $0<\lambda<\lambda_*$ and arbitrarily large uniformly $\lambda$-far $F$-partite $r$-graphs $P$ satisfying $\Hom_{\sigma^\circ}(H^\circ,P)\le\lambda^L v(P)^{v(H^\circ)}$ for every such pair.

Let $m=\lceil\lambda^{-\frac{1}{r-1}}\rceil$, let $Q$ be the disjoint union of $m$ copies of $P$, and let $\eps=\frac{\lambda}{m^{r-1}}$. Then $Q$ is uniformly $\eps$-far and $2^{-(r-1)}\lambda^2\le\eps\le\lambda^2$. Choose $P$ sufficiently large that $v(Q)\ge n_0(\eps)$ for the threshold in the $F$-abundance of $\cH$.

Fix $(H,\sigma)\in\cH$, let $K=H^\circ$, let $k=v(K)$, and let $a$ be the number of components of $K$. Every component of $K$ contains at least $r$ vertices, so $a\le\frac{k}{r}$. If $k\le B$, every component has $m$ choices for the copy of $P$ containing its image, and the isolated vertices of $H$ have at most $v(Q)^{v(H)-k}$ possible images. Hence
\[
\Emb_\sigma(H,Q)\le v(Q)^{v(H)-k}m^a\Hom_{\sigma^\circ}(K,P)\le m^{a-k}\lambda^L v(Q)^{v(H)}\le\lambda^L v(Q)^{v(H)}<\eps^A v(Q)^{v(H)}.
\]
If $k>B$, restricting an embedding of $H$ to $K$ and allowing arbitrary images for the isolated vertices gives $\Emb_\sigma(H,Q)\le v(Q)^{v(H)-k}\Hom(K,Q)$. Lemma~\ref{lemdisjoint} and $m\ge\lambda^{-\frac{1}{r-1}}$ give
\[
\Emb_\sigma(H,Q)\le\lambda^{\frac{k}{r}}v(Q)^{v(H)}\le\lambda^{\frac{B+1}{r}}v(Q)^{v(H)}<\eps^A v(Q)^{v(H)}.
\]
The two strict inequalities follow respectively from the first and second conditions on $\lambda_*$ together with $\eps^A\ge2^{-(r-1)A}\lambda^{2A}$. The third condition and $\eps\le\lambda^2$ give $\eps^{A-C}<c$. Hence every $(H,\sigma)\in\cH$ satisfies $\Emb_\sigma(H,Q)<c\eps^C v(Q)^{v(H)}$. This contradicts the $F$-abundance of $\cH$ because $Q$ is uniformly $\eps$-far and $v(Q)\ge n_0(\eps)$. Therefore some $F$-abundant $(H,\sigma)\in\cH$ satisfies $v(H^\circ)\le\lfloor2r(C+1)\rfloor$.

We next consider families of $r$-graphs. We apply Lemma~\ref{lemtensor} to all $F$-colourings of every graph $H^\circ$ with at most $B$ vertices and include the resulting sum over these colourings. Let $\cU$ be $F$-abundant with constants $c,C>0$, let $A=C+1$, let $L=2A+1$, and let $B=\lfloor2rA\rfloor$. Suppose for a contradiction that no $H\in\cU$ is both $F$-abundant and satisfies $v(H^\circ)\le B$. Choose $0<\lambda_*\le1$ such that
\[
\lambda_*^{L-2A}<2^{-(r-1)A}v(F)^{-B},\qquad \lambda_*^{\frac{B+1}{r}-2A}<2^{-(r-1)A},\qquad \lambda_*^{2(A-C)}<cv(F)^{-C}.
\]

If $H\in\cU$ and $v(H^\circ)\le B$, then every $F$-colouring $\tau$ of $H^\circ$ gives an $F$-coloured $r$-graph $(H^\circ,\tau)$ that is not $F$-abundant. Otherwise Lemmas~\ref{lemcolourtouncolour} and~\ref{lembasiccounts} would show that $H$ is $F$-abundant. Apply Lemma~\ref{lemtensor} to all the finitely many relabelled pairs that arise. We obtain $0<\lambda<\lambda_*$ and arbitrarily large uniformly $\lambda$-far $F$-partite $r$-graphs $P$ satisfying $\Hom_\tau(H^\circ,P)\le\lambda^L v(P)^{v(H^\circ)}$ for every such pair.

Let $m=\lceil\lambda^{-\frac{1}{r-1}}\rceil$, let $Q$ be the disjoint union of $m$ copies of $P$, let $\eps=\frac{\lambda}{m^{r-1}}$, and let $\delta=\frac{\eps}{v(F)}$. Lemmas~\ref{lemdisjoint} and~\ref{lemuniform} show that the $r$-graph obtained after the $F$-partition of $Q$ is forgotten is $\delta$-far from being $F$-free. We also have $2^{-(r-1)}\lambda^2\le\eps\le\lambda^2$. Choose $P$ sufficiently large that $v(Q)\ge n_0(\delta)$ for the threshold in the $F$-abundance of $\cU$.

Fix $H\in\cU$, let $K=H^\circ$, let $k=v(K)$, and let $a$ be the number of components of $K$. Every component of $K$ contains at least $r$ vertices, so $a\le\frac{k}{r}$. If $k\le B$, every embedding of $K$ into $Q$ determines one of at most $v(F)^k$ possible $F$-colourings. For each such colouring $\tau$, every component has $m$ choices for the copy of $P$ containing its image, and hence $\Hom_\tau(K,Q)=m^a\Hom_\tau(K,P)$. Allowing arbitrary images for the isolated vertices of $H$ gives
\[
\Emb(H,Q)\le v(F)^k m^{a-k}\lambda^L v(Q)^{v(H)}\le v(F)^B\lambda^L v(Q)^{v(H)}<\eps^A v(Q)^{v(H)}.
\]
If $k>B$, restricting an embedding of $H$ to $K$ gives $\Emb(H,Q)\le v(Q)^{v(H)-k}\Hom(K,Q)$. Lemma~\ref{lemdisjoint} and $m\ge\lambda^{-\frac{1}{r-1}}$ give $\Emb(H,Q)\le\lambda^{\frac{k}{r}}v(Q)^{v(H)}\le\lambda^{\frac{B+1}{r}}v(Q)^{v(H)}<\eps^A v(Q)^{v(H)}$. The first two conditions on $\lambda_*$ justify the two strict inequalities. The third condition and $\eps\le\lambda^2$ give $\eps^{A-C}<cv(F)^{-C}$, and hence $\eps^A<c\delta^C$. Therefore every $H\in\cU$ satisfies $\Emb(H,Q)<c\delta^C v(Q)^{v(H)}$. This contradicts the $F$-abundance of $\cU$. We conclude that some $F$-abundant $H\in\cU$ satisfies $v(H^\circ)\le\lfloor2r(C+1)\rfloor$.
\end{proof}

\begin{remark}\label{remcoreorder}
Theorem~\ref{thmboundedcore} cannot bound $v(H)$. For each nonnegative integer $s$, let $H_s$ be the $r$-graph on $[r+s]$ whose only edge is $[r]$, and let $E_r$ be the $r$-graph consisting of one edge. If an $n$-vertex $r$-graph $G$ is $\eps$-far from being $E_r$-free, then $e(G)\ge\eps n^r$. Each edge of $G$ gives $r!$ possible images of the edge $[r]$, and the $s$ isolated vertices have $(n-r)_s$ possible distinct images outside that edge. Hence $\Emb(H_s,G)=r!e(G)(n-r)_s$. For every fixed $s$ and all sufficiently large $n$, we have $(n-r)_s\ge\frac{1}{2}n^s$, and therefore $\Emb(H_s,G)\ge\frac{1}{2}\eps n^{r+s}$. Thus every $H_s$ is $E_r$-abundant with $c=\frac{1}{2}$ and $C=1$, while $v(H_s)=r+s$ is unbounded and $v(H_s^\circ)=r$.
\end{remark}

\begin{remark}\label{remisolatedF}
The assertion of Theorem~\ref{thmcompact} for families of $r$-graphs remains valid when $F$ has isolated vertices and $e(F)>0$. For every $r$-graph $G$ with at least $v(F)$ vertices, every copy of $F^\circ$ extends to a copy of $F$. Consequently a set of edges destroys every copy of $F$ if and only if it destroys every copy of $F^\circ$. The definitions of $F$-abundance and $F^\circ$-abundance for a family of $r$-graphs are therefore equivalent after their order thresholds are increased if necessary, and the assertion follows by applying Theorem~\ref{thmcompact} to $F^\circ$. The assumption that $F$ has no isolated vertices remains in the assertion for families of $F$-coloured $r$-graphs because Lemma~\ref{lemuniform} gives the required lower bound only for parts indexed by vertices that belong to an edge of $F$.
\end{remark}

\section{Applications}\label{secapplications}

This section has three parts. First, we associate a translation-invariant linear system with an abundant coloured $r$-graph. Secondly, we apply the bounded-core theorem to property testing, obtaining a tester based on one fixed subgraph and a concrete application to planarity. The last subsection discusses several examples showing the limits of these applications.

\subsection{Linear systems arising from abundant uniform hypergraphs}\label{subseclinearapplications}

We give an arithmetic consequence of the hypergraph compactness theorem. The construction is a higher uniformity version of the Ruzsa--Szemer\'edi construction. It associates a fixed translation-invariant linear system with every coloured abundant hypergraph. Theorem~\ref{thmboundedcore} then bounds the number of variables in a system selected from an infinite abundant family.

A homogeneous linear system $Bx=0$ is \emph{translation-invariant} if
$B\mathbf{1}=0$. A solution
$x=(x_1,\ldots,x_m)$ is \emph{proper} if its coordinates are pairwise
distinct.
We use the fixed $r$ and $F$ from the preceding sections and relabel $V(F)=[f]$. Let $(H,\sigma)$ be an $F$-coloured finite $r$-graph with no isolated vertices. We write $h=v(H)$ and $m=e(H)$. Choose pairwise distinct integers $\theta_1,\ldots,\theta_f$ and define $w_i=(1,\theta_i,\theta_i^2,\ldots,\theta_i^{r-1})\in\mathbb{Z}^r$ for $i\in[f]$.

On every edge $e\in E(H)$, the colouring $\sigma$ assigns distinct colours to its vertices. Thus the values $\theta_{\sigma(u)}$, $u\in e$, are distinct, and the corresponding vectors $w_{\sigma(u)}$ form an invertible Vandermonde matrix.
Define the rational linear map $D=D_{H,\sigma,\theta}\colon\mathbb{Q}^{V(H)}\rightarrow\mathbb{Q}^{E(H)}$ by
\[
(Dy)_e=
\sum_{u\in e}
\frac{y_u}{\displaystyle\prod_{v\in e\setminus\{u\}}
(\theta_{\sigma(u)}-\theta_{\sigma(v)})}
\qquad (e\in E(H)).
\]

The coordinate $(Dy)_e$ is the coefficient of $X^{r-1}$ in the unique polynomial of degree at most $r-1$ whose value at $\theta_{\sigma(u)}$ is $y_u$ for every $u\in e$. Let
\(\cS=\im_{\mathbb{Q}}(D),s=\dim_{\mathbb{Q}}(\mathcal{S})\).
Choose a full-row-rank $B\in\mathbb{Z}^{(m-s)\times m}$ such that
\(\ker_{\mathbb{Q}}(B)=\cS\).
 This matrix is obtained by choosing a rational basis for the space of all linear equations that vanish on $\cS$, and then clearing the denominators.
When $s=m$, this is the matrix with zero rows and $m$ columns. 
For a sufficiently large prime $p$, let $B\bmod p$ be the matrix obtained
by reducing every entry of $B$ modulo $p$.
Let $D_p$ be the reduction of $D$ modulo $p$.

\begin{theorem}\label{thmvandermondebridge}
Assume that there are constants $c,C>0$ such that, for every
$0<\eps\le1$, every sufficiently large uniformly $\eps$-far
$F$-partite $r$-graph $G$ satisfies
$\Emb_{\sigma}(H,G)\ge c\eps^C v(G)^h$.
For the matrix $B$ fixed above, there is a constant $c_0>0$ such that,
for every $\delta>0$, the following statements hold.

\begin{enumerate}
\item[(i)] There is $p_0=p_0(\delta)$ such that, for every prime
$p\ge p_0$ and every $A\subseteq\mathbb F_p$ with $|A|\ge\delta p$,
\[
|\ker(B\bmod p)\cap A^m|
\ge c_0\delta^C p^s.
\]

\item[(ii)] There is $N_0=N_0(\delta)$ such that, for every
$N\ge N_0$ and every $A\subseteq[N]$ with $|A|\ge\delta N$,
\[
|\{x\in A^m:Bx=0\text{ over }\mathbb Z\}|
\ge c_0\delta^C N^s.
\]
\end{enumerate}

If $m\ge2$, both statements remain valid when only proper solutions are
counted. Moreover, $B\mathbf 1=0$.
\end{theorem}

\begin{proof}
We first prove~(i). Fix $\delta>0$. Since we have fixed $D$, $B$, and the integers
$\theta_1,\ldots,\theta_f$ before, only finitely many primes can make a
denominator in $D$ vanish, a relevant Vandermonde determinant vanish, or the
rank of $D$ or $B$ decrease after reduction modulo $p$. 
Thus we take $p$
sufficiently large so that none of these cases occurs.

Fix $A\subseteq\mathbb F_p$ with $|A|\ge\delta p$, and write
$\delta_A=|A|/p$. Let $T_A=\mathbb F_p^{r-1}\times A$. We denote elements of
$T_A$ by $t=(z_0,z_1,\ldots,z_{r-2},a)$, where $a\in A$.
We construct an $F$-partite $r$-graph $G_A$. Its part indexed by
$i\in V(F)$ is $V_i=\{i\}\times\mathbb F_p$. For every $\eta\in E(F)$ and
$t\in T_A$, let
$
E_{\eta,t}=\{(i,w_i\cdot t):i\in\eta\}
$
be an edge of $G_A$.

The hypergraph $G_A$ is simple. Suppose that
$E_{\eta,t}=E_{\eta',t'}$. The part indices imply that
$\eta=\eta'$. Let $W_\eta$ be the matrix whose rows are the vectors
$w_i$ with $i\in\eta$. Equality of the second coordinates implies
$W_\eta t=W_\eta t'$. Moreover,
$
\det(W_\eta)
=
\pm\prod_{\substack{i,j\in\eta\\i<j}}
(\theta_j-\theta_i)
\ne0
\pmod p.
$
Thus $W_\eta$ is invertible over $\mathbb F_p$, and hence $t=t'$.
It follows that the map $(\eta,t)\mapsto E_{\eta,t}$ is injective.

Every $t\in T_A$ defines a copy $F_t$ of $F$ in $G_A$. The vertex
$i\in V(F)$ is mapped to $(i,w_i\cdot t)\in V_i$. Suppose that $F_t$ and $F_{t'}$ share an edge.
Then $E_{\eta,t}=E_{\eta,t'}$ for some $\eta\in E(F)$. By the preceding
argument, $t=t'$. Thus the copies $F_t$ are pairwise edge-disjoint.
For a fixed vertex $(i,y)\in V_i$, it belongs to $F_t$ exactly when
$\sum_{j=0}^{r-2}\theta_i^j z_j+\theta_i^{r-1}a=y$.
For every $a\in A$ and every choice of
$z_1,\ldots,z_{r-2}$, this equation uniquely determines $z_0$.
When $r=2$,
there are no freely chosen $z_j$. Hence every vertex of $G_A$ belongs to
exactly
$
|A|p^{r-2}=\delta_Ap^{r-1}
$
members of the edge-disjoint family $\{F_t:t\in T_A\}$.

Since $v(G_A)=fp$, this number is
$
\frac{\delta_A}{f^{r-1}}v(G_A)^{r-1}.
$
Thus $G_A$ is uniformly $\delta_A/f^{r-1}$-far from being $F$-free. Since
$\delta_A\ge\delta$, it is also uniformly
$\delta/f^{r-1}$-far. We take $p$ large enough that $v(G_A)=fp$ is above
the abundance threshold for this fixed value of the distance parameter.

We now count colour-preserving homomorphisms from $H$ to $G_A$.  Every vector
$y=(y_u)_{u\in V(H)}\in\mathbb F_p^{V(H)}$ defines a colour-preserving
vertex map
$
\psi_y(u)=(\sigma(u),y_u).
$

Fix an edge $e\in E(H)$. The matrix whose rows are
$w_{\sigma(u)}$, with $u\in e$, is invertible. Hence there is a unique
$t_e\in\mathbb F_p^r$ such that
$
y_u=w_{\sigma(u)}\cdot t_e
,(u\in e).
$
Equivalently, $t_e$ is the coefficient vector of the unique polynomial
$P_e$ of degree at most $r-1$ satisfying
$P_e(\theta_{\sigma(u)})=y_u$ for every $u\in e$. The Lagrange formula for
the leading coefficient of $P_e$ is the defining formula for
$(D_py)_e$. Thus the last coordinate of $t_e$ is $(D_py)_e$.
It follows that the image of $e$ under $\psi_y$ is an edge of $G_A$ exactly
when $(D_py)_e\in A$. Therefore,
\[
\psi_y\text{ is a colour-preserving homomorphism from $H$ to $G_A$}
\quad\Longleftrightarrow\quad
D_py\in A^m.
\]

Let $\cS_p=\im(D_p)$. By the choice of $p$, we have $\rank(D_p)=s$. Every
element of $\cS_p$ therefore has exactly $p^{h-s}$ preimages under $D_p$.
Hence
$
\Hom_{\sigma}(H,G_A)
=
p^{h-s}|\cS_p\cap A^m|.
$
The abundance hypothesis concerns embeddings. Since every embedding is a
homomorphism, we have
$\Hom_{\sigma}(H,G_A)\ge\Emb_{\sigma}(H,G_A)$.

By the abundance hypothesis with distance parameter
$\delta/f^{r-1}$, we obtain
$
|\cS_p\cap A^m|
\ge
c f^{h-(r-1)C}\delta^C p^s.
$
The reductions of $B$ and $D$ have the same ranks as the original matrices.
Moreover, $BD=0$, so
$\im(D_p)\subseteq\ker(B\bmod p)$. Both spaces have dimension $s$. It follows
that
$
\ker(B\bmod p)=\cS_p.
$
After decreasing the fixed positive constant, this proves the first lower
bound in~(i).

We next consider proper solutions. Let $e,e'\in E(H)$ be distinct. Since
$H$ is simple, some vertex lies in one of these edges and not in the other.
By the definition of $D$, it follows that the $e$-th and $e'$-th coordinates of
$Dy$ are different linear functions of $y$. Hence the linear map
$x\mapsto x_e-x_{e'}$ is nonzero on $\cS$.

By increasing the lower bound on $p$, this map remains nonzero on $\cS_p$.
Its kernel in $\cS_p$ has $p^{s-1}$ elements. Therefore the number of vectors
in $\cS_p$ with two equal coordinates is at most
$\binom{m}{2}p^{s-1}$. For fixed $\delta$, this is at most one half of the
preceding lower bound when $p$ is sufficiently large. After decreasing the
constant once more, the same bound holds when only proper solutions are
counted. This proves~(i).

We now prove~(ii). We first suppose that $B$ has no rows. Then $s=m$, and all
the vectors in $A^m$ are solutions. Note that $C\ge m$; indeed, fix
$0<\gamma\le1$ and choose sets $A_p\subseteq\mathbb F_p$ with
$|A_p|=\lceil\gamma p\rceil$. By~(i),
$
|A_p|^m\ge c_0\gamma^Cp^m.
$
After division by $p^m$ and letting $p$ tend to infinity, we obtain
$\gamma^m\ge c_0\gamma^C$. Since this holds for every
$0<\gamma\le1$, it follows that $C\ge m$.

Now let $A\subseteq[N]$ with $|A|\ge\delta N$. Since $C\ge m$ and
$0<\delta\le1$, we have
$
|A|^m\ge\delta^mN^m\ge\delta^CN^m.
$
When $m\ge2$, at most $\binom{m}{2}|A|^{m-1}$ tuples have two equal
coordinates. For sufficiently large $N$, at least one half of the tuples in
$A^m$ are proper. Thus~(ii) also holds for proper solutions after decreasing
$c_0$.

We may now assume that $B$ has at least one row. Let
$
L=
\max_{1\le j\le m-s}
\sum_{e=1}^m|B_{j,e}|,
$
and choose an integer $K>2L+2$. By Bertrand's postulate, for every
sufficiently large $N$ there is a prime $p$ such that
$KN<p<2KN$.

For sufficiently large $N$, this prime is large enough for~(i) with density
parameter $\delta/(2K)$. Regard $A\subseteq[N]$ as a subset of
$\mathbb F_p$. Since $p<2KN$, its density in $\mathbb F_p$ is greater than
$\delta/(2K)$. By~(i),
\[
|\{x\in A^m:Bx=0\pmod p\}|
\ge
c_0\left(\frac{\delta}{2K}\right)^Cp^s.
\]

For every $x\in[N]^m$ and every row $B_j$ of $B$, we have
$|B_jx|\le LN<p$. Hence
\[
B_jx=0\pmod p
\quad\Longleftrightarrow\quad
B_jx=0\text{ over }\mathbb Z.
\]
Since $p>KN$, the preceding lower bound is at least
$c_1\delta^CN^s$ for a fixed constant $c_1>0$. After decreasing $c_0$, this
proves the integer lower bound in~(ii).
When $m\ge2$, we use the proper form of~(i). Since $p>N$, distinct elements
of $[N]$ remain distinct modulo $p$. Thus the same argument proves~(ii) for
proper solutions.

It remains to prove that the system is translation-invariant. Set
$y_u=\theta_{\sigma(u)}^{r-1}$ for every $u\in V(H)$. On every edge, these
values are evaluations of the polynomial $X^{r-1}$. Its leading coefficient
is one. By the definition of $D$, we have $Dy=\mathbf 1$. Thus
$\mathbf 1\in\cS=\ker_{\mathbb Q}(B)$, and hence $B\mathbf 1=0$.
\end{proof}

The system in Theorem~\ref{thmvandermondebridge} need not contain a nonzero
equation. The following condition ensures that $B$ has a nonzero row.

\begin{proposition}\label{propnonemptyvandermondesystem}
Let $q$ be the number of components of $H$.
Then
$\rank(B)=\codim(\cS)\ge m-h+q(r-1)$.
In particular, $B$ has a nonzero row whenever
$m\ge h-q(r-1)+1$.
\end{proposition}

\begin{proof}
Fix a component $J$ of $H$ and a polynomial
$P$ with rational coefficients and degree at most $r-2$. Define $y_u=P(\theta_{\sigma(u)})$ for $u\in J$ and $y_u=0$ outside $J$. On every edge contained in $J$, the interpolation polynomial has degree at
most $r-2$. On every other edge, all the corresponding values are zero.
Thus the coefficient of $X^{r-1}$ is zero on every edge, and hence
$y\in\ker(D)$. The map from $P$ to $y$ is injective. Every component contains an edge, and a polynomial of degree at most $r-2$ which vanishes at the $r$ distinct values associated with that edge is the zero polynomial. The subspaces obtained from different components have disjoint supports. It follows that $\dim(\ker(D))\ge q(r-1)$. Therefore $s=\rank(D)\le h-q(r-1)$, and the stated codimension bound follows from $\rank(B)=m-s$.
\end{proof}

We next combine the bridge with Theorem~\ref{thmboundedcore}.

\begin{corollary}\label{corboundedarithmeticwitness}
Suppose that a coloured or uncoloured family is $F$-abundant
with constants $c,C>0$. 
Then we can choose $H_0$ from the uncoloured family,
or $(H_0,\sigma)$ from the coloured family.
We can also choose an $F$-colouring $\sigma_0$ of $H_0^\circ$. In the
coloured case, let
$\sigma_0=\sigma|_{V(H_0^\circ)}$. In the uncoloured case, choose
$\sigma_0$ by Corollary~\ref{corfixedcolour}. These choices have the following
properties.

\begin{enumerate}
\item[(i)] We have
$v(H_0^\circ)\le\lfloor2r(C+1)\rfloor$, and
$(H_0^\circ,\sigma_0)$ is $F$-abundant.

\item[(ii)] The linear system associated with
$(H_0^\circ,\sigma_0)$ is translation-invariant. The finite-field and
integer lower bounds in Theorem~\ref{thmvandermondebridge} hold for this
system. It has at most
$\binom{\lfloor2r(C+1)\rfloor}{r}$
variables. If it has at least two variables, the same bounds hold for proper
solutions.

\item[(iii)] If $H_0^\circ$ satisfies the condition in
Proposition~\ref{propnonemptyvandermondesystem}, then the system contains a
nonzero equation.
\end{enumerate}
\end{corollary}

\begin{proof}
By Theorem~\ref{thmboundedcore}, we can choose $H_0$ with
$v(H_0^\circ)\le\lfloor2r(C+1)\rfloor$. The $r$-graph $H_0^\circ$ is
$F$-abundant. Indeed, every embedding of $H_0$ restricts to an
embedding of $H_0^\circ$, and each embedding of $H_0^\circ$ has at most
$v(G)^{v(H_0)-v(H_0^\circ)}$ extensions to a map from $H_0$ to $G$.

In the coloured case, we restrict the given colouring to $H_0^\circ$. In the
uncoloured case, by Corollary~\ref{corfixedcolour}, there is an $F$-colouring
$\sigma_0$ for which $(H_0^\circ,\sigma_0)$ is abundant.
Since $H_0^\circ$ is simple, it has at most
$\binom{v(H_0^\circ)}{r}$ edges. Thus its associated system has at most
$\binom{\lfloor2r(C+1)\rfloor}{r}$ variables. The remaining conclusions
follow from Theorem~\ref{thmvandermondebridge} and
Proposition~\ref{propnonemptyvandermondesystem}.
\end{proof}

When $r=2$ and the family consists of cycles, the system has one equation and a stronger conclusion follows from the additive results of Gir\~ao, Hurley, Illingworth and Michel.
We need the following theorems.
For a translation-invariant equation
$E\colon a_1x_1+\cdots+a_kx_k=0$ with nonzero integer coefficients, its
\emph{genus} is the largest integer $g$ for which $[k]$ can be partitioned
into nonempty sets $I_1,\ldots,I_g$ such that
$\sum_{i\in I_j}a_i=0$ for every $j$; see
\cite[Definition~3.1]{GHIM}. We write $R_E(N)$ for the largest size of a set
$A\subseteq[N]$ containing no proper solution to $E$.

\begin{theorem}[Gir\~ao, Hurley, Illingworth and Michel {\cite[Theorem~1.2]{GHIM}}]\label{thmGHIMsqrt}
Let $E$ be a translation-invariant linear equation of genus at least two. Then
$
R_E(N)=O_E(\sqrt N),(N\ge1) 
$
\end{theorem}

\begin{theorem}[Gir\~ao, Hurley, Illingworth and Michel {\cite[Theorem~4.3]{GHIM}}]\label{thmGHIMcyclegenus}
Let $F$ be a graph, and let $(C_\ell,\sigma)$ be an $F$-coloured cycle which is $F$-abundant. Write the vertices of $C_\ell$ as $v_1,\ldots,v_\ell$, with $v_{\ell+1}=v_1$. Then, for every integer $Q\ge |V(F)|$ and every injection $a\colon V(F)\to[Q]$, the equation
$
\sum_{i=1}^{\ell}
\bigl(a(\sigma(v_{i+1}))-a(\sigma(v_i))\bigr)x_i=0
$
has genus at least two.
\end{theorem}

\begin{corollary}\label{corboundedgenuscyclewitness}
Let $\cC$ be a family of $F$-coloured cycles which is 
$F$-abundant with constants $c,C>0$. Then there is a member
$(C_\ell,\sigma)\in\cC$ with the following properties.

\begin{enumerate}
\item[(i)] The cycle $(C_\ell,\sigma)$ is $F$-abundant, and
$\ell\le\lfloor4(C+1)\rfloor$.

\item[(ii)] Write the vertices of $C_\ell$ as
$v_1,\ldots,v_\ell$, where $v_{\ell+1}=v_1$. For every integer $Q\ge f$
and every injection $a\colon V(F)\to[Q]$, the equation
$
\sum_{i=1}^{\ell}
\bigl(a(\sigma(v_{i+1}))-a(\sigma(v_i))\bigr)x_i=0
$
has genus at least two.

\item[(iii)] For every fixed equation $E$ in~(ii), we have
$R_E(N)=O_E(\sqrt N)$.
\end{enumerate}

If $\cC$ is an uncoloured cycle family, then there are a cycle
$C_\ell\in\cC$ and an $F$-colouring $\sigma$ of $C_\ell$ for which
{\rm(i)}--{\rm(iii)} hold.
\end{corollary}

\begin{proof}
By Theorem~\ref{thmboundedcore} with $r=2$, we obtain an
abundant cycle of length at most $\lfloor4(C+1)\rfloor$. This proves~(i).
By Theorem~\ref{thmGHIMcyclegenus}, every equation in~(ii) has genus at
least two. By Theorem~\ref{thmGHIMsqrt}, each fixed equation $E$ in~(ii)
satisfies $R_E(N)=O_E(\sqrt N)$. This proves~(ii) and~(iii).
For an uncoloured cycle family, Theorem~\ref{thmboundedcore} first provides
the cycle. By Corollary~\ref{corfixedcolour}, it has a fixed abundant
$F$-colouring. The same argument proves~(i)--(iii).
\end{proof}

We emphasize the direction of Theorem~\ref{thmvandermondebridge}. Coloured hypergraph abundance implies a polynomial lower bound for the number of solutions of the associated linear system. We do not use or assert the converse. Such a converse is open even for general coloured graphs. When $r\ge3$, the associated system may contain several equations, so the single-equation genus used in Corollary~\ref{corboundedgenuscyclewitness} has no direct general analogue.

\subsection{Property testing}\label{sectesting}

We retain the fixed $r$ and $F$ from the preceding sections. Let $\cP$ be a
property of finite simple $r$-graphs, that is, a class of $r$-graphs closed
under isomorphism. We write $\cP_n$ for the members of $\cP$ with $n$
vertices, and assume that $\cP_n\neq\varnothing$ for every $n$.

For an $n$-vertex $r$-graph $G$, we set
$\dist_{\cP}(G)=
\min\{|E(G)\mathbin{\triangle}E(G')|/n^r:
V(G')=V(G),\ G'\in\cP_n\}$.
We use the dense adjacency-query model~\cite{GGR98}. An adjacency query
specifies an $r$-set $e\subseteq V(G)$ and asks whether $e\in E(G)$. For
$0<\eps\leq1$, a one-sided $\eps$-tester for $\cP$ is a randomized algorithm
with adjacency-query access to an $n$-vertex $r$-graph $G$. It always accepts
when $G\in\cP_n$, and it rejects with probability at least $2/3$ when
$\dist_{\cP}(G)\geq\eps$. No condition is required when
$0<\dist_{\cP}(G)<\eps$. The query complexity is the maximum number of
adjacency queries made by the tester.

By Theorem~\ref{thmboundedcore}, every $F$-abundant family with
constants $c,C>0$ contains an $F$-abundant member $H_0$ such
that $v(H_0^\circ)\leq\lfloor2r(C+1)\rfloor$. We use this fixed graph
together with a blow-up argument.

The next lemma removes the order threshold in the definition of the
$F$-abundance of a fixed graph.
It first proves a homomorphism bound. The embedding bound holds
when the order is large enough in terms of the distance parameter.

\begin{lemma}\label{lemblowupamplification}
Let $H$ be $F$-abundant with constants $c,C>0$, and let
$h=v(H)$. Suppose that $0<\eta\leq1$ and that $G$ is a finite $r$-graph
with $\dist_F(G)\geq\eta$. Then the following statements hold.

\begin{enumerate}
\item[(i)] We have
$\Hom(H,G)\geq c\eta^C v(G)^h$.

\item[(ii)] If $v(G)=n$ and
$n\geq2\binom{h}{2}/(c\eta^C)$, then
$\Emb(H,G)\geq(c/2)\eta^C n^h$.
\end{enumerate}
\end{lemma}

\begin{proof}
Let $n=v(G)$. The balanced $t$-blow-up $G[t]$ is obtained by replacing each
vertex $v\in V(G)$ with the cluster $C_v=\{v\}\times[t]$. Thus
$V(G[t])=V(G)\times[t]$. An $r$-set is an edge of $G[t]$ exactly when it has
one vertex in each of $r$ distinct clusters $C_{v_1},\ldots,C_{v_r}$ and
$\{v_1,\ldots,v_r\}$ is an edge of $G$.

We first show that $G[t]$ is $\eta$-far from being $F$-free.
Let $D\subseteq E(G[t])$ and suppose that $G[t]-D$ is $F$-free. Choose one
vertex independently and uniformly from each cluster. Before the edges of
$D$ are removed, the selected vertices induce a copy of $G$. After these
edges are removed, the resulting $r$-graph is $F$-free, since it is a
subgraph of $G[t]-D$. Since $G$ is $\eta$-far from being $F$-free, at least
$\eta n^r$ edges of the selected copy belong to $D$.
Each edge of $D$ belongs to the selected copy with probability $t^{-r}$.
By expectation, $|D|/t^r\geq\eta n^r$. Hence
$|D|\geq\eta(tn)^r$, and $G[t]$ is $\eta$-far from being $F$-free.

Choose $t$ so that $tn$ is above the abundance threshold for $\eta$. By
 abundance,
$\Emb(H,G[t])\geq c\eta^C t^h n^h$.
Every embedding of $H$ into $G[t]$ projects to a homomorphism from $H$ to
$G$. For a fixed homomorphism $\phi:H\to G$, an embedding into $G[t]$ which
projects to $\phi$ is determined by choosing one of the $t$ copies of
$\phi(u)$ for each vertex $u\in V(H)$. Hence there are at most $t^h$ such
embeddings.
Thus
$t^h\Hom(H,G)\geq\Emb(H,G[t])$, and we obtain
$\Hom(H,G)\geq c\eta^C n^h$. This proves~(i).

Every non-injective homomorphism identifies a pair of vertices of $H$.
Therefore
$\Hom(H,G)-\Emb(H,G)\leq\binom{h}{2}n^{h-1}$.
By~(i) and the order assumption,
\[
\Emb(H,G)
\geq
c\eta^C n^h-\binom{h}{2}n^{h-1}
\geq
\frac{c}{2}\eta^C n^h.
\]
This proves~(ii).
\end{proof}

We now obtain a one-sided tester for every input order. The comparison between
the two distances is assumed for every order.

\begin{theorem}\label{thmallordertester}
Let $\cP$ be a property of finite simple $r$-graphs. Let $\cW$ be a family
of finite $r$-graphs, each with at least one edge. Suppose that every
$r$-graph containing a member of $\cW$ lies outside $\cP$. Suppose also that
$\cW$ is $F$-abundant with constants $c,C>0$.
Assume that membership in $\cP$ is decidable and there are constants $a,b>0$, with $a\leq1$, such that every
finite $r$-graph $G$ satisfies
$\dist_F(G)\geq a\dist_{\cP}(G)^b$.
Then there are a graph $W_0\in\cW$ and constants $\kappa,D>0$.
Let $K=W_0^\circ$. These objects have the following properties.

\begin{enumerate}
\item[(i)]
$v(K)\leq\lfloor2r(C+1)\rfloor$.

\item[(ii)] For every $0<\eps\leq1$ and $n$, there is a one-sided $\eps$-tester for $\cP$ on
$n$-vertex inputs.
It uses at most
$\kappa\eps^{-D}$ adjacency queries.

\item[(iii)] For every fixed $\eps$, on all sufficiently large inputs the
tester rejects only after finding a copy of the fixed graph $K$.
\end{enumerate}

\end{theorem}

\begin{proof}
By Theorem~\ref{thmboundedcore}, choose an $F$-abundant graph
$W_0\in\cW$ such that its non-isolated core $K$ has order
$k\leq\lfloor2r(C+1)\rfloor$. This proves~(i).

The graph $K$ is also $F$-abundant. Indeed, every embedding of
$W_0$ restricts to an embedding of $K$. Each embedding of $K$ into an
$n$-vertex graph has at most $n^{v(W_0)-k}$ extensions to an embedding of
$W_0$. Let $c_0,C_0>0$ be abundance constants for $K$. Decrease $c_0$ if
necessary so that $c_0\leq1$, and write $h_0=v(W_0)$. Since $W_0$ has at
least one edge, we have $k\geq r$.

Fix $0<\eps\leq1$. Set
$
\eta=a\eps^b,
\rho_\eps=\frac{c_0}{2}\eta^{C_0},
N_\eps=
\max\left\{
h_0,k,
\left\lceil
\frac{2\binom{k}{2}}{c_0\eta^{C_0}}
\right\rceil
\right\}.
$
Suppose first that $n<N_\eps$. The tester queries all
$\binom{n}{r}$ possible edges. It then decides whether the input belongs to
$\cP$.
Now suppose that $n\geq N_\eps$. The tester performs
$T=\lceil\rho_\eps^{-1}\log3\rceil$ independent trials. In each trial, it
chooses an ordered $k$-tuple uniformly from $V(G)^k$. A tuple with a repeated
vertex is discarded. Otherwise, the tester queries the images of all edges
of $K$. It rejects if all these images are edges of $G$. If no trial finds a
copy of $K$, then it accepts.

We first check that every graph in $\cP$ is accepted. Suppose that
$G\in\cP$ and $n\geq h_0$. Then $G$ contains no copy of $K$. Otherwise, we
could add $h_0-k$ distinct vertices to this copy and obtain a copy of
$W_0$. This is possible since the remaining vertices of $W_0$ are isolated.
It contradicts the assumption on $\cW$. Thus the sampling part always accepts every graph in $\cP$. When
$n<N_\eps$, the tester decides membership in $\cP$ exactly and therefore
also accepts every graph in $\cP$.

We next consider a graph $G$ with $\dist_{\cP}(G)\geq\eps$. By the distance
assumption, $\dist_F(G)\geq\eta$. If $n\geq N_\eps$, then by
Lemma~\ref{lemblowupamplification} 
$\Emb(K,G)\geq\rho_\eps n^k$. Hence one trial finds an embedding of $K$ with
probability at least $\rho_\eps$. The probability that all trials fail is at
most
\[
(1-\rho_\eps)^T
\leq
\exp(-\rho_\eps T)
\leq
\frac13.
\]
Thus the tester rejects every graph which is $\eps$-far from $\cP$ with
probability at least $2/3$.

It remains to count the queries. There is a constant $M$, depending only on
the fixed graphs and constants, such that
$N_\eps\leq M\eps^{-bC_0}$. When $n<N_\eps$, querying the whole input uses at most
$\binom{n}{r}\leq n^r\leq M^r\eps^{-bC_0r}$
adjacency queries. The sampling part uses at most
$
e(K)
\left(
2c_0^{-1}a^{-C_0}\eps^{-bC_0}\log3+1
\right)
$
queries. Since $r\geq2$ and $0<\eps\leq1$, both bounds are at most
$\kappa\eps^{-bC_0r}$ after increasing $\kappa$. Thus~(ii) holds with
$D=bC_0r$.

For every fixed $\eps$, the sampling part is used whenever
$n\geq N_\eps$. It rejects only after finding a copy of $K$. This
proves~(iii).
\end{proof}

A subdivision of a graph is obtained by replacing each edge with a path,
where the internal vertices of the replacement paths are all distinct.
Let $\cP_{\mathrm{pl}}$ be the property of planar graphs. Let
$\cW_{\mathrm{Kur}}$ be the family of all subdivisions of $K_5$ and
$K_{3,3}$, where each graph is also regarded as a subdivision of itself.
By Kuratowski's theorem, a graph is planar if and only if it contains no
member of $\cW_{\mathrm{Kur}}$.

\begin{corollary}\label{corplanaritytesting}
The following statements hold.

\begin{enumerate}
\item[(i)] The family $\cW_{\mathrm{Kur}}$ contains a
$K_2$-abundant graph with at most $40$ vertices. The graph $K_{3,3}$ itself
is an explicit choice.

\item[(ii)] Planarity has a one-sided $\eps$-tester using
$O(\eps^{-18})$ adjacency queries for every number of vertices $n$.

\item[(iii)] There is an absolute constant $c>0$ such that, when
$n\geq c\eps^{-9}$, the tester uses $O(\eps^{-9})$ adjacency queries and
rejects only after finding a copy of $K_{3,3}$.
\end{enumerate}
\end{corollary}

\begin{proof}
Let $G$ be an $n$-vertex graph and write $\gamma=2e(G)/n^2$. For
$x_1,x_2,x_3\in V(G)$, let
$c(x_1,x_2,x_3)=|N(x_1)\cap N(x_2)\cap N(x_3)|$. By H\"older's inequality,
\[
\begin{aligned}
\Hom(K_{3,3},G)
&=\sum_{x_1,x_2,x_3}c(x_1,x_2,x_3)^3 \\
&\geq
\frac{\left(\sum_{x_1,x_2,x_3}c(x_1,x_2,x_3)\right)^3}{n^6} \\
&=
\frac{\left(\sum_{v\in V(G)}d(v)^3\right)^3}{n^6}
\geq
\gamma^9n^6.
\end{aligned}
\]

If $\dist_{K_2}(G)\geq\eta$, then
$e(G)\geq\eta n^2$ and $\gamma\geq2\eta$. At most
$\binom{6}{2}n^5$ homomorphisms identify two vertices. Thus
$\Emb(K_{3,3},G)\geq\eta^9n^6$ when $n$ is sufficiently large in terms of
$\eta$. Hence $K_{3,3}$ is $K_2$-abundant with exponent $9$.
Since $K_{3,3}\in\cW_{\mathrm{Kur}}$, the family
$\cW_{\mathrm{Kur}}$ is $K_2$-abundant with the same exponent.

By Theorem~\ref{thmboundedcore}, the family $\cW_{\mathrm{Kur}}$ contains an
abundant member with at most $4(9+1)=40$ vertices. The preceding
count identifies the explicit choice $K_{3,3}$. This proves~(i).

Deleting all edges from a graph produces a planar graph. Therefore every
$n$-vertex graph $G$ satisfies
$\dist_{K_2}(G)=e(G)/n^2\geq\dist_{\cP_{\mathrm{pl}}}(G)$.
Thus Theorem~\ref{thmallordertester} applies with $a=b=1$.

For the tester, we use the fixed graph $K_{3,3}$. Its abundance exponent can
be taken as $9$. The large-order part therefore has order threshold
$O(\eps^{-9})$ and uses $O(\eps^{-9})$ queries. When $n$ is below the order threshold, the tester queries all
$\binom{n}{2}$ pairs. Since this threshold is $O(\eps^{-9})$, the number of
adjacency queries is $O(\eps^{-18})$. This proves~(ii) and~(iii).
\end{proof}
\subsection{Limitations}
Polynomial one-sided testability does not imply the existence of a finite
forbidden family. Bipartiteness has a polynomial-query one-sided tester in the
dense graph model~\cite{GGR98}. However, no finite family $\mathcal Q$ of
non-bipartite graphs has the property that every graph far from bipartite
contains a member of $\mathcal Q$.

Let $L=\max\{v(H)\mid H\in\mathcal Q\}$, and choose an odd integer $m>L$.
Let $B_m(s)$ be the balanced blow-up of $C_m$ with vertex classes of order
$s$. We first verify that $B_m(s)$ has odd girth $m$. Suppose that it contains
an odd cycle of length $\ell<m$. The projection of this cycle onto $C_m$ is a
closed walk of length $\ell$. Orient $C_m$, and assign $+1$ to every forward
step and $-1$ to every backward step. Since the walk is closed, the sum of
these signs is divisible by $m$. The sum is odd and has absolute value at most
$\ell<m$, which is impossible. Thus $B_m(s)$ contains no odd cycle of length
less than $m$. Every graph $H\in\mathcal Q$ contains an odd cycle of length at
most $v(H)\le L<m$. It follows that $B_m(s)$ contains no member of
$\mathcal Q$.

There are $s^m$ transversal copies of $C_m$ in $B_m(s)$, and every edge lies
in exactly $s^{m-2}$ of them. Let $D\subseteq E(B_m(s))$ and suppose that
$B_m(s)-D$ is bipartite. The set $D$ meets every transversal copy of $C_m$.
By double counting the pairs $(e,C)$ with $e\in D\cap E(C)$, we obtain
$|D|s^{m-2}\ge s^m$. Hence $|D|\ge s^2$. Since
$v(B_m(s))=ms$, the graph $B_m(s)$ is $m^{-2}$-far from bipartite.

The compactness argument uses ordinary subgraph copies. The corresponding
statement for induced copies is false, even when $F=K_2$. Consider the family
$\{K_3,\overline{K_3}\}$. By $R(3,3)=6$, every set of six vertices contains a
set of three vertices inducing either $K_3$ or $\overline{K_3}$. Double
counting pairs $(S,T)$, where $|S|=6$, $T\subseteq S$, and
$G[T]\in\{K_3,\overline{K_3}\}$, shows that every $n$-vertex graph contains at
least
$\binom{n}{6}/\binom{n-3}{3}=\binom{n}{3}/20$
such vertex triples. At least one of $K_3$ and $\overline{K_3}$ therefore
occurs as an induced subgraph on at least $\binom{n}{3}/40$ vertex triples.
Thus the family is abundant for induced copies.

Neither $K_3$ nor $\overline{K_3}$ is abundant for induced
copies. Balanced complete bipartite graphs are a constant distance from being
$K_2$-free and contain no induced $K_3$. Complete graphs are also a constant
distance from being $K_2$-free and contain no induced $\overline{K_3}$.

The compactness theorem asserts the existence of an abundant member, but it
does not provide a general method for finding one. We now show that this
cannot be avoided.

We say that a Turing machine enumerates a family $\cA$ of finite graphs if it
outputs every member of $\cA$, possibly more than once, and outputs no graph
outside $\cA$.

\begin{theorem}\label{thmundecidableabundance}
There is no algorithm that takes as input a Turing machine enumerating a
family $\cA$ of finite simple graphs and decides whether $\cA$ is
$K_3$-abundant.
\end{theorem}

\begin{proof}
Let $M$ be a Turing machine. We construct a Turing machine $E_M$ which
enumerates a family $\cA_M$. At stage $j\geq1$, the machine $E_M$ outputs
$K_4\mathbin{\dot\cup}I_j$, where $I_j$ is the graph with $j$ vertices and
no edges. It also simulates one further step of $M$. If $M$ halts, then
$E_M$ also outputs $K_2$.

Suppose first that $M$ halts. Then $K_2\in\cA_M$. If an $n$-vertex graph
$G$ is $\eps$-far from being $K_3$-free, then
$e(G)\geq\eps n^2$, since deleting all edges makes $G$ $K_3$-free. Hence
$\Emb(K_2,G)=2e(G)\geq2\eps n^2$. Thus $K_2$ is $K_3$-abundant, and
therefore $\cA_M$ is $K_3$-abundant.

Now suppose that $M$ does not halt. Then
$\cA_M=\{K_4\mathbin{\dot\cup}I_j:j\geq1\}$. For every positive integer
$s$, let $T_s$ be the complete tripartite graph with vertex classes
$A=\{a_0,\ldots,a_{s-1}\}$,
$B=\{b_0,\ldots,b_{s-1}\}$, and
$C=\{c_0,\ldots,c_{s-1}\}$.
For every $0\leq i,j<s$, consider the triangle
$\{a_i,b_j,c_{(i+j)\bmod s}\}$. No two of these $s^2$ triangles share an
edge. Indeed, an edge between any two of the three vertex classes determines
the remaining index. Thus every set of edges whose deletion makes $T_s$
$K_3$-free contains at least $s^2$ edges. Since $v(T_s)=3s$, we have
$\dist_{K_3}(T_s)\geq1/9$.

The graph $T_s$ is tripartite, so it contains no copy of $K_4$. Hence it
contains no member of $\cA_M$. Since the orders of the graphs $T_s$ are
unbounded, the family $\cA_M$ is not $K_3$-abundant.
We have proved that $\cA_M$ is $K_3$-abundant if and only if $M$ halts.
Therefore an algorithm deciding whether $\cA_M$ is $K_3$-abundant would
decide whether $M$ halts. This contradicts the undecidability of the halting
problem.
\end{proof}

By Theorem~\ref{thmcompact}, a family is $K_3$-abundant exactly when it
contains a $K_3$-abundant member. Thus no general algorithm can always decide
whether such a member exists and return one when it does. In the testing
results above, the graph $W_0$ and its abundance constants are fixed before
the tester is used. They are not computed from a Turing machine enumerating
the family $\cW$.

\section*{Declaration of the Use of Generative AI Tools} Generative AI tools were used during the development of Lemma~\ref{lemtensor}, specifically to explore the product construction used in its proof and to assist in checking the parameter choices and resulting homomorphism estimates. They were also used to proofread the manuscript and polish its English presentation. All AI-generated suggestions incorporated into the manuscript were independently checked and revised by the authors. The authors take full responsibility for the final content of the paper.

\bibliographystyle{plain} \bibliography{refs_final}

\end{document}